\documentclass[12pt]{article}
\usepackage[english]{babel}
\usepackage{amsfonts,amsmath,amsxtra,amsthm,amssymb,latexsym}
\usepackage{hyperref}

\usepackage[usenames,dvipsnames]{color}
\usepackage{verbatim}

\definecolor{light-gray}{gray}{0.50}
\newtheorem{theorem}{Theorem}[section]
\newtheorem{corollary}[theorem]{Corollary}
\newtheorem{lemma}[theorem]{Lemma}

\newtheorem{definition}[theorem]{Definition}
\newtheorem{remark}[theorem]{Remark}

\newtheorem{proposition}[theorem]{Proposition}

\numberwithin{equation}{section}

\def\ker{{\rm ker\,}}
\def\dom{{\rm dom\,}}
\def\max{{\rm max\,}}
\def\min{{\rm min\,}}

\def\dim{{\rm dim\,}}

 \newcommand{\cH}{{\mathcal H}} \def\Ext{{\rm Ext\,}}
\newcommand{\gG}{{\Gamma}}
\newcommand{\cC}{{\mathcal C}}
\newcommand{\gotH}{{\mathfrak H}}

 \DeclareMathOperator{\spann}{span} \DeclareMathOperator{\imm}{Im}
  \DeclareMathOperator{\oo}{o}
\def\Ext{{\rm Ext\,}}

\title{To the Spectral Theory  of the Bessel Operator on Finite Interval and Half-Line}

\author{Aleksandra~Yu.~Ananieva$^1$, Viktoriya~S.~Budika$^2$
\\{\ }\\
$^1$ Institute of Applied Mathematics and Mechanics  \\
R. Luxemburg  str. 74, Donetsk  \\
E-mail: ananeva89@gmail.com \\
$^2$ Institute of Applied Mathematics and Mechanics\\
R. Luxemburg  str. 74, Donetsk \\
E-mail: budyka.vik@gmail.com}
\date{}

\begin{document}
\maketitle
\begin{abstract}

\noindent

The minimal and maximal operators generated by the Bessel differential expression on the
finite interval and a half-line are studied. All non-negative self-adjoint extensions of
the minimal operator are described. Also we  obtain a description of the domain of the
Friedrichs extension of the minimal operator in the framework of extension theory of
symmetric operators by applying the technique of boundary triplets and the corresponding
Weyl functions, and by using the quadratic form method.
\end{abstract}

\section{Introduction}
The one-dimensional Bessel differential expression was investigated in the classical
form
\begin{equation}\label{E:1.1}
\tau_\nu=-\frac{d^2}{dx^2}+\frac{\nu^2-\frac{1}{4}}{x^2},\qquad
\nu\in[0,1)\setminus\{1/2\}
\end{equation}
on the half-line $\mathbb R_+$ in numerous papers. Here, the parameter $\nu$ is the
order of the Bessel functions involved. When $\nu=\frac{1}{2}$, it is the regular case.
In particular, some results of spectral analysis were investigated in
works~\cite{BDG,EvKalf, Ful, FulLag,KST}. We especially mention papers of W.N.~Everitt
and H.~Kalf~\cite{EvKalf,Kalf} the most relevant to our interest. Here, Titchmarsh--Weyl
$m$--coefficient was explicitly computed in $L^2(\mathbb R_+)$ using the classical
definition. From the Nevanlinna representation of this $m$--coefficient the spectral
function $\Sigma$ was obtained to describe the spectrum of the associated self-adjoint
operator in $L^2(\mathbb R_+)$. The additional analysis then yields the limit behaviour
of the functions in the domain of the Friedrichs extension (see L. Bruneau, J.
Derezi\'{n}ski and V. Georgescu \cite{BDG}, W.N.~Everitt and H.~Kalf~\cite{EvKalf,Kalf})
and Krein extension (see \cite{BDG}).

 In this paper, we consider Bessel operator~\eqref{E:1.1}. Under the above restriction
 ($\nu\in[0,1)$) the endpoint 0 of the equation
\begin{equation}\label{eq111}
-y''(x)+\frac{\nu^2-\frac{1}{4}}{x^2}y(x)=\lambda y(x)
\end{equation}
is the singular limit-circle case, with respect to $L^2(\mathbb R_+)$ or $L^2(0,b)$,
except for the regular case.

We study the minimal and maximal Bessel operators on a finite interval and a half-line.
We prove that the domain of the minimal operator ${A_{\nu,\infty}}_\min$ associated with
$\tau_\nu$ in $L^2(\mathbb R_+)$ is given by
\begin{equation}\label{E1}
\dom({A_{\nu,\infty}}_\min)={H}_0^2(\mathbb R_+),
\end{equation}
and we prove similar formula for the operator on a finite interval.

We investigate spectral properties of the Bessel operator by applying the technique of
boundary triplets and corresponding Weyl functions. This new approach to extension
theory of symmetric operators developed during last three decades (see
\cite{DM_91,DerMal95,GG} and references in therein).

We construct a boundary triplet for the maximal operators in $L^2(\mathbb R_+)$ and
$L^2(0,b)$ and compute the corresponding Weyl functions. We determine the domains of
Friedrichs and Krein extensions. In addition, all self-adjoint and all nonnegative
self-adjoint extensions of the minimal Bessel operator are described. We also obtain the
Weyl functions on half-line as a limit of corresponding Weyl functions of the operator
considered in the finite interval. In particular, we obtain other proofs of results of
L. Bruneau, J. Derezi\'{n}ski and V. Georgescu \cite{BDG}, W.N.~Everitt and
H.~Kalf~\cite{EvKalf,Kalf}.

 \section{Preliminaries}

\subsection{Boundary triplets and self-adjoint extension.}
In this section we briefly review the notion of abstract boundary triplets in the
extension theory of symmetric operators.

Let $A$ be a closed densely defined symmetric operator in the separable  Hilbert space
$\mathfrak H $ with equal deficiency indices $n_\pm(A)=\dim\ker(A^*\pm iI)\leq\infty$.

\begin{definition}[{\cite{GG}}]\label{D:3.1}
  \,\,A totality $\Pi=\{\mathcal H,\Gamma_0,\Gamma_1\}$ is called a {\rm boundary
triplet} for the adjoint operator $A^*$ of $A$ if $\mathcal H$ is an auxiliary Hilbert
space and
$\Gamma_0,\Gamma_1:\  \dom(A^*)\rightarrow\mathcal H$ are linear mappings such that\\
 $(i)$ the following abstract second  Green identity holds
      \begin{equation}\label{E:3.2}
      (A^*f,g)-(f,A^*g)=(\Gamma_1f,\Gamma_0g)_{\mathcal{H}}-(\Gamma_0f,\Gamma_1g)_{\mathcal{H}},\qquad f,g\in\dom(A^*);
    \end{equation}
    $(ii)$ the mapping $\Gamma:=(\Gamma_0,\Gamma_1)^\top: \dom(A^*)
\rightarrow  \mathcal H \oplus\mathcal H$ is surjective.
\end{definition}
First note that a boundary triplet 
for $A^*$ exists since the deficiency indices of $A$ are assumed to be
equal. Noreover, $\mathrm{n}_\pm(A) = \dim(\cH)$ and $A=A^*\upharpoonright\left(\ker(\Gamma_0) \cap \ker(\Gamma_1)\right)$ hold. Note also that a boundary triplet for $A^*$ is not unique.

A closed extension $\widetilde{A}$ of $A$ is called \emph{proper}
if $A\subseteq\widetilde{A}\subseteq A^*$. Two proper extensions
$\widetilde{A}_1$ and $\widetilde{A}_2$ of $A$ are called
\emph{disjoint} if
$\dom(\widetilde{A}_1)\cap\dom(\widetilde{A}_2)=\dom(A)$ and
\emph{transversal} if in addition
$\dom(\widetilde{A}_1)\dotplus\dom(\widetilde{A}_2)=\dom(A^*)$.
The set of all proper extensions of $A$ is denoted by $\Ext A.$

With   a  boundary   triplet   $\Pi=\{\mathcal H,\Gamma_0,\Gamma_1\}$ for $A^*$  one
associates two self-adjoint extensions $A_j:=A^*\upharpoonright \ker (\Gamma_j), \;
j\in\{0,1\}.$

\begin{proposition}[{\cite{DM_91,GG}}]\label{propo}
Let  $\Pi=\{\mathcal{H},\Gamma_0,\Gamma_1\}$  be   a boundary triplet for $A^*$. Then
the  mapping
\begin{equation}\label{ext}
\Ext_A\ni\widetilde{A}:=A_\Theta\rightarrow
\Theta:=\Gamma(\dom(\widetilde{A}))=\big\{\{\Gamma_0f,\Gamma_1f\}:\,\,f\in
\dom(\widetilde{A})\big\}
\end{equation}
establishes  a  bijective  correspondence   between  the  set  of all  closed proper
extensions $\Ext_A$  of $A$ and  the set of all closed   linear  relations
$\widetilde{\mathcal{C}}(\mathcal{H})$ in $\mathcal{H}$. Furthermore,  the following
assertions  hold.
\begin{itemize}
\item[\textit{(i)}]
  The  equality $(A_\Theta)^*=A_{\Theta^*}$
%
%
holds  for any  $\Theta\in\widetilde{\mathcal{C}}(\mathcal{H})$.

\item[\textit{(ii)}]
 The extension $A_\Theta$  in \eqref{ext} is  symmetric
(self-adjoint)  if  and only  if  $\Theta$  is symmetric (self-adjoint).
\item[\textit{(iii)}]
  If,  in  addition,    extensions $A_\Theta$ and
$A_0$ are disjoint, i.e., $\dom(A_\Theta)\cap\dom(A_0)=\dom(A)$, then \eqref{ext} takes
the form
\begin{equation}\label{extensB}
A_\Theta=A_B=A^*\!\upharpoonright\ker\bigl(\Gamma_1-B\Gamma_0\bigr),\quad
B\in\mathcal{C}(\mathcal{H}).
   \end{equation}
\end{itemize}
 \end{proposition}
\subsection{Weyl functions and extension of nonnegative operator}

It is known that the classical  Weyl-Titchmarsh
functions play an important role in the direct and inverse
spectral theory of singular Sturm-Liouville operators.
In \cite{DM_91}  the concept of the classical Weyl--Titchmarsh $m$-function
from the theory of Sturm-Liouville operators  was generalized to
the case of symmetric operators  with equal deficiency indices.
The role of abstract Weyl functions in the extension theory  is
similar to that of the classical Weyl--Titchmarsh $m$-function in the spectral
theory of singular Sturm-Liouville operators.

Let $\frak N_z:=\ker(A^*-z)$ be the defect subspace of $A$.

\begin{definition}[{\cite{DM_91}}]\label{D:4.0}
Let $A$ be a densely defined closed symmetric operator in $\mathfrak{H}$ with equal
deficiency indices and let $\Pi=\{\cH,\gG_0,\gG_1\}$ be a boundary triplet for $A^*$.
The operator valued functions $\gamma :\rho(A_0)\rightarrow  [\cH,\mathfrak{H}]$ and
$M:\rho(A_0)\rightarrow [\cH]$ defined by
\begin{equation}\label{II.1.3_01}
\gamma(z):=\bigl(\Gamma_0\!\upharpoonright\mathfrak{N}_z\bigr)^{-1}
\qquad\text{and}\qquad M(z):=\Gamma_1\gamma(z), \qquad z\in\rho(A_0),
\end{equation}
are called the {\em $\gamma$-field} and the {\em Weyl function}, respectively,
corresponding to the boundary triplet $\Pi.$
\end{definition}
The $\gamma$-field $\gamma(\cdot)$ and the Weyl function $M(\cdot)$ in \eqref{II.1.3_01}
are well defined. 
Moreover, both $\gamma(\cdot)$ and $M(\cdot)$ are holomorphic on $\rho(A_0)$ and the
following relations hold (see \cite{DM_91})
\begin{gather}\label{II.1.3_02'}
\gamma(z)=\bigl(I+(z-\zeta)(A_0-z)^{-1}\bigr)\gamma(\zeta),
\\
%
%
%
%
\label{II.1.3_02}
M(z)-M(\zeta)^*=(z-\overline\zeta)\gamma(\zeta)^*\gamma(z), 
\\
   \label{II.1.3_02B}
\gamma^*(\overline z)
 = \Gamma_1(A_0 - z)^{-1}, \qquad z,\
\zeta\in\rho(A_0).
      \end{gather}
Identity \eqref{II.1.3_02}  yields that $M(\cdot)$ is an $R_{\cH}$-function (or {\it
Nevanlinna function}), that is, $M(\cdot)$ is an ($[\cH]$-valued) holomorphic function
on $\mathbb{C}\setminus \mathbb{R}$ and
     \begin{equation}\label{II.1.3_03}
 \imm z\cdot\imm M(z)\geq 0,\qquad  M(z^*)=M(\overline
z),\qquad  z\in \mathbb{C}\setminus \mathbb{R}.
\end{equation}
 Besides, it follows  from \eqref{II.1.3_02}
that  $M(\cdot)$ satisfies $0\in \rho(\imm M(z))$ for
$z\in\mathbb{C}\setminus\mathbb{R}$.
Since $A$ is densely defined, $M(\cdot)$ admits an integral representation (see, for
instance, \cite{DerMal95})
\begin{equation}\label{WF_intrepr}
M(z)=C_0+\int_{\mathbb{R}}\left(\frac{1}{t-z}-\frac{t}{1+t^2}\right)d\Sigma_M(t),\qquad
z\in\rho(A_0),
\end{equation}
where $\Sigma_M(\cdot)$ is an operator-valued Borel measure on $\mathbb{R}$ satisfying
$\int_\mathbb{R} \frac{1}{1 + t^2}d\Sigma_M(t) \in [\cH]$ and $C_0 = C_0^*\in [\cH]$.
The integral in (\ref{WF_intrepr}) is understood in the strong sense.
%
%

In contrast to spectral measures of self-adjoint operators the measure $\Sigma_M(\cdot)$
is not necessarily orthogonal. However, the measure $\Sigma_M$ is uniquely determined by
the Nevanlinna function $M(\cdot)$.  The operator-valued measure $\Sigma_M$ is called
\emph{the spectral measure} of $M(\cdot)$. If $A$ is a simple symmetric operator, then
the Weyl function $M(\cdot)$ determines the pair $\{A,A_0\}$ up to unitary equivalence
(see \cite{DerMal95}). Due to this fact,  spectral
properties of $A_0$ can be expressed in terms of $M(\cdot)$.\\
%
%
%
%
%
%
%
%
%
%
Assume that  a symmetric operator $A\in \cC(\mathfrak{H})$ is
nonnegative. Then the set $\Ext_A(0,\infty)$ of its  nonnegative
self-adjoint extensions is non-empty (see \cite{AG}).
Moreover, there is a maximal 
nonnegative extension $A_{\rm F}$ (also called \emph{Friedrichs'}
or \emph{hard} extension) and there is a  minimal nonnegative
extension $A_{\rm K}$ (\emph{Krein's} or \emph{soft} extension)
satisfying
$$
(A_F+x)^{-1} \le ( \widetilde{A} + x)^{-1} \le (A_K + x)^{-1}, \qquad x\in
(0,\infty), \quad \widetilde{A}\in  \Ext_A(0,\infty)
$$
(for detail we refer the reader to \cite{AG}).

The following proposition characterizes the Friedrichs and Krein extensions in terms of
the Weyl function.

\begin{proposition}[{\cite{DM_91, DerMal95}}]\label{prkf}
Let $A$ be  a densely defined nonnegative symmetric operator with finite deficiency
indices  in $\gotH$, and  let $\Pi=\{\cH,\gG_0,\gG_1\}$ be a boundary
 triplet for  $A^*$.
\, Let  also  $M(\cdot)$ be the corresponding Weyl function.  Then the  following
assertions hold.
\begin{itemize}
\item[\textit{(i)}] Extensions
$A_0$  and $A_K$  are  disjoint  ($A_0$ and $A_F$ are disjoint) if and only if
$$M(0)\in\cC(\cH) \qquad (M(-\infty)\in\cC(\cH),\ \  \mbox{respectively}).$$
Moreover,
\begin{align*}
\dom(A_K)&=\dom(A^*)\upharpoonright\ker(\Gamma_1-M(0)\Gamma_0)\\
(\dom(A_F)&=\dom(A^*)\upharpoonright\ker(\Gamma_1-M(-\infty)\Gamma_0),\ \
\mbox{respectively}).
\end{align*}
\item[\textit{(ii)}]  $A_0=A_K$   ($A_0=A_F$) if  and only  if
\begin{equation*}
 \begin{aligned} 
 \lim_{x\uparrow0}(M(x)f,f)&=+\infty,\quad f\in\cH\setminus\{0\} \\
 (\lim_{x\downarrow-\infty}(M(x)f,f)&=-\infty, \quad f\in\cH\setminus\{0\},
\quad \mbox{respectively}).
 \end{aligned}
\end{equation*}
 \item[\textit{(iii)}]   The set of all non-negative self-adjoint extensions of $A$
 admits
 parametrization
 \eqref{ext},  where $\Theta$  satisfies
\begin{equation}\label{kappa1}
 \Theta-M(0)\geq 0\qquad  (\Theta-M(-\infty)\leq 0,\,\, \mbox{respectively}).
 \end{equation}

   \end{itemize}
\end{proposition}
\subsection{Bessel functions}

Consider the equation

\begin{equation}\label{Bes_eq}
z^2\frac{d^2u}{dz^2}+z\frac{du}{dz}+(z^2-\nu^2)u=0.
\end{equation}
Solutions of the~\eqref{Bes_eq} are the Bessel functions of the first $J_{\pm\nu}$ and
second $Y_\nu$ kind, respectively (see~\cite[Ch. 9]{A&S},~\cite[App.
2]{AG},~\cite[p.~284--285]{Naj69}).

     Recall that the asymptotic behaviors of
the Bessel functions $J_{\nu}(t)$ and $J_{-\nu}(t)$ for $t\rightarrow0$ have the form

\begin{equation}\label{Bessel_0}
J_{\nu}(t)=\frac{t^{\nu}}{2^{\nu}\Gamma(1+\nu)}[1+O(t^2)], \qquad
J_{-\nu}(t)=\frac{2^{\nu}}{\Gamma(1-\nu)}t^{-\nu}[1+O(t^2)],\quad  t\to0,
\end{equation}
and the asymptotic behavior of the Bessel functions $Y_{\nu}(t)$ for $t\rightarrow0$ has
the form
\begin{equation}\label{Bessel_Y}
Y_{0}(t)=\frac{2}{\pi}\left(\log\left(\frac{t}{2}\right)+\gamma\right)\cdot[1+O(t^2)],
\qquad
Y_{\nu}(t)=-\frac{\Gamma(\nu)}{\pi}\left(\frac{2}{t}\right)^{\nu}\cdot[1+O(t^2)],\quad
t\to0,
\end{equation}
where $\gamma$ is Euler's constant.

 Moreover as $t\rightarrow\infty$ we have
 \begin{equation}\label{asym_infty}
 \left\{
   \begin{array}{l}
     J_{\nu}(t)=
 \sqrt{\frac{2}{\pi t}}\cos\left(t-\frac{\nu\pi}{2}-\frac{\pi}{4}\right)+O(|t|^{-\frac{3}{2}}), \\
     J_{-\nu}(t)=
 \sqrt{\frac{2}{\pi t}}\cos\left(t+\frac{\nu\pi}{2}-\frac{\pi}{4}\right)+O(|t|^{-\frac{3}{2}}), \\
   Y_{\nu}(t)=
 \sqrt{\frac{2}{\pi t}}\sin\left(t-\frac{\nu\pi}{2}-\frac{\pi}{4}\right)+O(|t|^{-\frac{3}{2}}), \\
   \end{array}
 \right.\qquad t\to\infty.
 \end{equation}

 Also, we need decomposition of the Bessel functions into Taylor series about zero
(see~\cite[Formulas 9.1.10, 9.1.12, 9.1.13]{A&S})
\begin{equation}\label{Jnu}
J_\nu(z)=\left(\frac{1}{2}z\right)^\nu\sum\limits_{k=0}^\infty\frac{\left(-\frac{1}{4}z^2\right)^k}
{k!\Gamma(\nu+k+1)},
\end{equation}
\begin{equation}\label{J0}
J_0(z)=1-\frac{\frac{1}{4}z^2}{(1!)^2}+\frac{\left(\frac{1}{4}z^2\right)^2}{(2!)^2}-
\frac{\left(\frac{1}{4}z^2\right)^3}{(3!)^2}+\ldots,
\end{equation}
\begin{equation}\label{Y0}
Y_0(z)=\frac{2}{\pi}\left\{\log\left(\frac{1}{2}z\right)+\gamma\right\}J_0(z)+\frac{2}{\pi}\left\{
\frac{\frac{1}{4}z^2}{(1!)^2}-\left(1+\frac{1}{2}\right)\frac{\left(\frac{1}{4}z^2\right)^2}{(2!)^2}+
\left(1+\frac{1}{2}+\frac{1}{3}\right)\frac{\left(\frac{1}{4}z^2\right)^3}{(3!)^2}-\ldots\right\}.
\end{equation}

We use the following properties of Bessel functions (see~\cite[Formula 9.1.28]{A&S})
\begin{equation}\label{Bessel'}
J'_0(t)=-J_1(t), \qquad Y'_0(t)=-Y_1(t).
\end{equation}
 Also recall~\cite[App. 2]{AG} that the Bessel function $Y_\nu$ of the second kind is given by
 \begin{equation}\label{E:Y_v}
Y_\nu(t)=\frac{J_\nu(t)\cos\pi\nu-J_{-\nu}(t)}{\sin\pi\nu}, \qquad \nu\neq0.
 \end{equation}

Next, we need formulas (see~\cite[Formula 9.1.29]{A&S})
\begin{eqnarray}\label{9.1.29}
zf_{\nu}^{'}(z)=\l qz^qf_{\nu-1}(z)+(p-\nu q)f_\nu(z),\nonumber\\
zf_{\nu}^{'}(z)=-\l qz^qf_{\nu+1}(z)+(p+\nu q)f_\nu(z),
\end{eqnarray}
in which $f_\nu(z)=z^pG_\nu(\l z^q)$ where $G_{\nu}(\cdot)$ is one of the Bessel
functions $J_{\nu}(\cdot)$, $Y_{\nu}(\cdot)$ or a linear combination, and $p$, $q$, $\l$
do not depend on $\nu$.

Applying formula \eqref{9.1.29} for $\l=1$, $q=1/2$, $p=0$ to the functions
$f_{\nu}=x^{1/2}G_{\nu}(x\sqrt{z})$ where $G_{\nu}(\cdot)$ is one of the Bessel
functions $J_{\nu}(\cdot)$, $Y_{\nu}(\cdot)$, we obtain
 \begin{equation}\label{E:reccur+}
 [f_{\nu},
 x^{1/2+\nu}]_x=\sqrt{z}x^{1/2+\nu}f_{\nu+1}, \qquad
  [f_{-\nu}, x^{1/2+\nu}]_x=-\sqrt{z}x^{1/2+\nu}f_{-\nu-1},
 \end{equation}
 and
\begin{equation}\label{E:reccur-}
[f_{\nu}, x^{1/2-\nu}]_x=-\sqrt{z}x^{1/2-\nu}f_{\nu-1},\qquad [f_{-\nu},
 x^{1/2-\nu}]_x=\sqrt{z}x^{1/2-\nu}f_{-\nu+1},
 \end{equation}
where $[f,g]_x:=f(x)\overline{{g}'(x)}-f'(x)\overline{{g}(x)}$, for all
$x\in\mathbb{R}_+.$

The general solution of the equation~\eqref{eq111} is given by
\begin{equation}
y(x;\lambda)=c_1x^{1/2}J_\nu(x\sqrt{\lambda})+c_2x^{1/2}Y_\nu(x\sqrt{\lambda}),
\end{equation}
where $c_1$, $c_2$ are arbitrary constants.

 \section{Bessel operator $S_{\nu;b}$ on the interval}
In what follows, we need the following auxiliary lemma.

\begin{lemma}[{\cite[p.~318--319]{Stein}}]\label{L:1.115}
Let $T_K$ be the operator in $L^p[0,\infty)$ of the form
\begin{equation}\label{E:1.109}
T_K: f\mapsto \int \limits_0^{\infty}K(x,t)f(t)dt,
\end{equation}
and its kernel $K(x,t)$ has a degree of homogeneity $-1$, i.e. $K(\lambda x, \lambda
t)=\lambda^{-1} K(x,t),$ $\lambda
>0$. Then the operator $T_K$ is bounded in $L^p[0,\infty)$ and its norm is
\begin{equation}\label{E:1.110}
\|T_K\|_p:=\|T_K\|_{L^p\rightarrow L^p}= \int \limits_0^{\infty}|K(1,t)|t^{-1/p}dy.
\end{equation}
\end{lemma}

Suppose further that $\mathcal I$ is the operator of integration, $\mathcal
I:f\mapsto\int\limits_0^xf(t)dt$. Then
\begin{equation}\label{E:I}
\mathcal I^2f=\int \limits_0^{x}(x-t) f(t)dt.
\end{equation}
Also assume that $Q: f\mapsto \frac{1}{x^2}f(x)$.

\begin{lemma}\label{L:1.2}
 The operator $Q\mathcal I^2$
\begin{equation}\label{E:1.100}
Q\mathcal I^2:f\mapsto\frac{1}{x^2} \int \limits_0^{x}(x-t) f(t)dt,
\end{equation}
is bounded in $L^2[0, b]$ for each $b \in (0, \infty]$, and $\|Q\mathcal I^2 \|_2=
\frac{4}{3}.$
\end{lemma}
\begin{proof}
 Let
 \begin{equation}\label{E:1.102}
  K(x,t)= \left\{
\begin{array}{ll}
\frac{1}{x}\left(1-\frac{t}{x}\right), & t \leqslant x,\\
0, & t > x.\\
\end{array}
\right.
\end{equation}
Noting that $K(\lambda x, \lambda t)=\lambda^{-1} K(x,t)$ and applying
Lemma~\ref{L:1.115} to the operator $T_K=Q\mathcal I^2$, we obtain
\begin{equation}\label{E:1.111}
\|Q\mathcal I^2\|_2 = \|T_K\|_2 = \int \limits_0^{\infty}|K(1,t)|t^{-1/2}dt = \int
\limits_0^{1}(1-t)t^{-1/2}dt = \frac{4}{3}.
\end{equation}
\end{proof}
Let $H^2[0,b]$ be the Sobolev space on $[0,b]$. Assume also

\begin{equation}\label{E:1.125}
\widetilde{H}_0^{2}[0, b]=\{f \in H^2[0, b]: f(0)=f'(0)=0\}.
\end{equation}

\begin{lemma}\label{L:1.222}
If $f\in \widetilde{H}_0^2[0,1]$, then the following relations hold:
\begin{equation}\label{E:4.01}
  f(x)=\oo(x^{3/2}),\qquad
  f'(x)=\oo(x^{1/2})\qquad\text{as}\ x\to0.
\end{equation}
\end{lemma}
\begin{proof}
Since $f\in\widetilde{H}_0^2[0,1]$, then $f'(x)=\int\limits_0^xf''(t)dt$. Therefore by
the Cauchy--Bunyakovsky inequality
\begin{equation}\label{E:4.02}
  |f'(x)|^2\leqslant\left(\int\limits_0^x|f''(t)|dt\right)^2\leqslant x
  \int\limits_0^x|f''(t)|^2dt=  x\cdot \oo(1)=\oo(x)\qquad\text{as}\ x\to0,
\end{equation}
i.e. $ f'(x)=\oo(x^{1/2})$, which proves the second estimate in \eqref{E:4.01}.

 Further, since $f\in\widetilde{H}_0^2[0,1]$, we get $f(x)=\int\limits_0^xf'(t)dt$. Hence,
\begin{equation}\label{E:4.03}
  |f(x)|\leqslant\int\limits_0^x|f'(t)|dt\leqslant
  \int\limits_0^x \oo(x^{1/2})dx = \oo(x^{3/2})\quad  \text{as}\  x\to 0.
\end{equation}
The first estimate in \eqref{E:4.01} is proved.
\end{proof}

Let $D_{\min}^2$ be a minimal differential operator of the 2nd order, generated in
$L^2(0,
 b)$ by differential expression $-d^2/dx^2,$
     \begin{equation}\label{E:1.127}
\dom(D_{\min}^2)=H^2_0[0,b]=\{f \in H^2[0, b]: f(0)=f'(0)=f(b)=f'(b)=0\}.
\end{equation}

Let ${S_{\nu,b}}:={S_{\nu,b}}_{\min}$ and ${S_{\nu,b}}_\max$ are the minimal and maximal
operators, respectively, generated by the differential expression~\eqref{E:1.1} in
$L^2(0, b),\ b<\infty$.

 \begin{theorem}\label{P:1.3}
Let $\nu \in [0,1)$. Then the following assertions hold.
\begin{itemize}
  \item[(i)] The operator ${S_{\nu,b}}$ is a non-negative and its deficiency indices
  are $n_{\pm}({S_{\nu,b}})=2$.
\item[(ii)] The domain of the operator ${S_{\nu,b}}$ is given by $$\dom({S_{\nu,b}})=H_0^{2}[0, b].$$
\item[(iii)] ${S_{\nu,b}}_{\max}={S^*_{\nu,b}}$ and
\begin{equation}\label{E:1.118}
\dom({S^*_{\nu,b}})= \left\{
\begin{array}{ll}
  \widetilde{H}_0^{2}[0,b]\dotplus\spann \{x^{1/2+\nu}, x^{1/2-\nu}\}, & \nu \in(0,1), \\
  \widetilde{H}_0^{2}[0,b]\dotplus\spann \{x^{1/2}, x^{1/2}\log(x)\}, & \nu=0.
\end{array}
\right.
\end{equation}
\end{itemize}
\end{theorem}
\begin{proof}
 {(i)--(ii)} 
%
The function $u \in \widetilde{H}_0^{2}[0,b]$  admits the integral representation
$u(x)=\int \limits_0^{x}(x-t) u''(t)dt.$ Therefore,
\begin{equation}\label{E:1.11}
Q u(x)=\frac{1}{x^2}u(x)=\frac{1}{x^2}\int \limits_0^{x}(x-t) u''(t)dt=(Q\mathcal
I^2(D_{\min}^2u))(x).
\end{equation}
 By virtue of Lemma~\ref{L:1.2}, this yields
\begin{eqnarray}\label{E:300.43}
  \|Q u\|_{2}&=&\left\|\frac{1}{x^2}u\right\|_{2}=\left\|Q\mathcal I^2D_{\min}^2 u\right\|_{2}
  \leqslant \|Q\mathcal I^2\|_2 \cdot\|D_{\min}^2 u\|_{2} \nonumber\\
&=&\frac{4}{3}\|D_{\min}^2u\|_{2} \leqslant\frac{4}{3}\|u\|_{H^2_0[0,b]}.
     \end{eqnarray}
It is easy to see that $\nu^2-\frac{1}{4}$ admits the representation
$\nu^2-\frac{1}{4}=\frac{3}{4}(1-\varepsilon)$, where $\varepsilon >0.$ Then
relation~\eqref{E:300.43} implies the estimate
\begin{equation}\label{E:300.44}
\left\|\left(\nu^2-\frac{1}{4}\right)
  Qu\right\|_{2}=\left|\nu^2-\frac{1}{4}\right|\cdot\|Qu\|_{2}\leqslant\frac{3}{4}(1-\varepsilon)
  \cdot\frac{4}{3}\|u\|_{H^2_0[0,b]}
  =(1-\varepsilon)\|u\|_{H^2_0[0,b]}, \quad u\in H^2_0[0,b].
\end{equation}
 Estimate~\eqref{E:300.44} means that $Q$ is strongly $D_{\min}^2$-bounded. Therefore,
 by the Kato--Rellich theorem (see \cite{Ka})
$n_{\pm}({S_{\nu,b}})=n_{\pm}(D_{{\min}}^2)=2$ and $\dom({S_{\nu,b}})=H_0^{2}[0,b]$.

{{(iii)}} Since
$$
\tau_\nu x^{1/2\pm\nu}=0,
$$
where the equality is understood in the meaning of the theory of distributions, and
$x^{1/2\pm\nu} \in L^2(0, b)$, then
$$\{x^{1/2+\nu}, x^{1/2-\nu}\} \subset\dom({S_{\nu,b}}_{\max})=\dom({S^*_{\nu,b}}),$$
and  $\ker({S^*_{\nu,b}})=\{x^{1/2+\nu}, x^{1/2-\nu}\}$ $\subset L^2(0, b)$. In
addition, it is clear that $\widetilde{H}_0^2[0,b] \subset \dom({S^*_{\nu,b}})$ and
$\dim(\widetilde{H}_0^2[0,b])/$ $\dom({S_{\nu,b}}))=2.$ On the other hand, since
$n_{\pm}({S_{\nu,b}})=2,$ we have \linebreak$\dim(\dom({S^*_{\nu,b}})/\dom({S_{\nu,b}}))
= 2n_{\pm}({S_{\nu,b}})=4$ by the first Neumann formula. Therefore,
formula~\eqref{E:1.118} is valid.
    \end{proof}

 Consider the quadratic form $\mathfrak{s}'_{\nu;b}$ associated with the operator
  ${S_{\nu,b}},$
\begin{equation}\label{E:1.131}
\mathfrak{s}'_{\nu;b}[u]:=({S_{\nu,b}}u, u), \qquad
\dom(\mathfrak{s}'_{\nu;b})=\dom({S_{\nu,b}})=H^2_0[0,b].
\end{equation}

It is clear that $S_{1/2,b}=-D_{\min}^2$.
\begin{theorem}\label{P:1.8}
Let $\nu\in[0,1)$ and ${S_{\nu,b}}_F$ be the Friedrichs extension of the operator
${S_{\nu,b}}$. Also assume $\xi\in C_0^1[0,b]$ such that $\xi(x)=1$ for $x\in[0,b/2]$
and $\xi(b)=0$. Then:

 $(i)$ For $\nu\in(0,1)$ the quadratic form $\mathfrak{s}_{\nu,b}$ quadratic form
 associated with  the Friedrichs extension ${S}_{{\nu,b}_F}$ takes the form
\begin{equation}\label{E:1.100}
\mathfrak{s}_{\nu,b}[u]=\int \limits_0^b|u'(x)|^2dx+\left(\nu^2-\frac{1}{4}\right)\int
\limits_0^b\frac{|u(x)|^2}{x^2}dx,
\end{equation}
\begin{equation}\label{E:1.100a}
\dom(\mathfrak{s}_{\nu,b}) = H^1_0[0,b].
\end{equation}

$(ii)$ For $\nu=0$ the quadratic form $\mathfrak{s}_{0,b}$ quadratic form
 associated with  the Friedrichs extension ${S}_{{0,b}_F}$ takes the form
\begin{equation}\label{E:1.100y}
\mathfrak{s}_{0,b}[u]=\int \limits_0^b\left|u'(x)-\frac{u(x)}{2x}\right|^2dx,
\end{equation}

\begin{equation}\label{E:1.100aa}
 \dom(\mathfrak{s}_{0,b})\supset H^1_0[0,b]\dot{+}
                                     \mathrm{span}\{x^{\frac{1}{2}}(x-b),x^{\frac{1}{2}}\xi(x)\}.
\end{equation}
Wherein $\dim\left(\mathrm{dom}(\mathfrak{s}_{0,b})\diagup {H}_0^1[0,b]\right)=\infty.$

$(iii)$ The domain of the Friedrichs extension ${S_{\nu,b}}_F$ of the operator
${S_{\nu,b}}$
 takes the form
\begin{equation}\label{fr_ext}
\dom({S_{\nu,b}}_F)=\left\{
\begin{array}{ll}
  H_0^{2}[0,b]\dotplus\spann\{x^{1/2+\nu}(x-b), x^2(x-b)\}, & \nu\in(0,1), \\
  H_0^{2}[0,b]\dotplus\spann\{x^{1/2}(x-b), x^{1/2}\xi(x)\}, & \nu=0.
\end{array}
\right.
\end{equation}
\end{theorem}
\begin{proof}
(i) By Hardy's inequality for $\nu\in(0,1)$ and $u\in H_0^1[0,b]$
\begin{equation}\label{E:Kvad_form11}
\frak s_{\nu,b}[u]=\|u'(t)\|_{2}^2+
(\nu^2-1/4)\int\limits_0^b\frac{|u(t)|^2}{t^2}dt\leqslant \|u'(t)\|_{2}^2(1+|4\nu^2-1|).
\end{equation}
Thus $H_0^1[0,b]\subset\dom(\mathfrak{s}_{\nu,b})$.

We prove the converse inequality. Suppose first that $\nu\in[1/2,1)$. Then
\begin{equation}\label{E:Har_11}
\frak s_{\nu,b}[u]=\|u'(t)\|_{2}^2+
(\nu^2-1/4)\int\limits_0^b\frac{|u(t)|^2}{t^2}dt\geqslant \|u'(t)\|_{2}^2, \quad u\in
H_0^1[0,b].
\end{equation}

If $\nu\in(0,1/2)$, then for $u\in H_0^1[0,b]$ applying the Hardy's inequality we obtain
\begin{equation}\label{E:Kvad_form2}
\frak s_{\nu,b}[u]=\|u'(t)\|_{2}^2- (1/4-\nu^2)\int\limits_0^b\frac{|u(t)|^2}{t^2}dt
\geqslant \|u'(t)\|_{2}^2+(4\nu^2-1)\|u'(t)\|_{2}^2=4\nu^2\|u'(t)\|_{2}^2.
\end{equation}

So on $H_0^1[0,b]$ the energy norm of ${S_{\nu,b}}$ is equivalent to the norm of space
$H_0^1[0,b]$. Since $H_0^2[0,b]=\dom({S_{\nu,b}})$ is dense in the energy space of the
operator ${S_{\nu,b}}$, then $\dom(\frak s_{\nu,b})$ and $H_0^1[0,b]$ coincide
algebraically and topologically.

(ii) Let $u_1(x)=x^{1/2}(x-b)$ and $u_2(x)=x^{1/2}\xi(x)$ then
$$
\mathfrak{s}_{0,b}[u_1]=\int \limits_0^b xdx<\infty,\qquad\quad
\mathfrak{s}_{0,b}[u_2]=\int \limits_0^b x(\xi'(x))^2dx<\infty.
$$
So $\{x^{1/2}(x-b),x^{1/2}\xi(x)\}\subset\dom(\mathfrak{s}_{0,b})$.

 (iii) We note that $H_0^2[0,b]\subset H_0^1[0,b]$. If $u(x)=x^{1/2+\nu}(x-b)$
then $u'(\cdot)\in L^2(0,b)$, but $u(\cdot)\not\in\dom({S_{\nu,b}})=H_0^{2}[0, b]$. By
the construction of the Friedrichs extension and the equalities \eqref{E:1.118}, we
obtain
        \begin{equation*}
\begin{gathered}
\dom({S_{\nu,b}}_F)=\dom({S^*_{\nu,b}})\cap
\dom(\mathfrak{s}_{\nu,b})=\dom({S^*_{\nu,b}})
\cap H^1_0[0,b]=\\
=H_0^{2}[0, b]\dotplus \spann\{x^{1/2}(x-b), x^{1/2}\xi(x)\}.
\end{gathered}
\end{equation*}

The case $\nu=0$ is considered similarly.
\end{proof}

The case $\nu\in[0,1/\sqrt{2})$ in Proposition~\ref{P:1.8} can be treated by means of
KLMN--theorem.  Therefore, applying Hardy's inequality for  one gets
      \begin{equation}\label{E:1.300}
\left|\left(\nu^2-\frac{1}{4}\right)\int \limits_0^b\frac{|u(x)|^2}{x^2}dx\right|
\leqslant 4\left|\nu^2-\frac{1}{4}\right|\int \limits_0^{b}|u'|^2dx\leqslant
(1-\varepsilon)\mathfrak{t}_{D^2_{\min}}[u],\qquad u\in H^1_0[0,b].
   \end{equation}

Hence, the form $\left(\nu^2-\frac{1}{4}\right)\mathfrak{q}$ is strongly
$\mathfrak{t}_{D^2_{\min}}$-bounded, where $\frak q[u]:=\int
\limits_0^b\frac{|u(x)|^2}{x^2}dx$. By the KLMN-theorem~\cite{Ka}
$\dom(\mathfrak{s}_{\nu,b})=\dom(\mathfrak{t}_{D^2_{\min}})=H^1_0[0,b].$

This argument was already used in our previous paper~\cite{AA}.


\section{Bessel operator ${A_{\nu,b}}$ on the interval}

Here, we consider the Bessel operator ${A_{\nu,b}}$ generated by the differential
expression~\eqref{E:1.1} in $L^2(0, b)$  with the domain
\begin{equation}\label{Dirihle}
\dom({A_{\nu,b}})=\{f\in \mathrm{dom}(S^*_{\nu,b}):f(0)=f'(0)=f(b)=0\},\qquad
\nu\in[0,1).
\end{equation}
\begin{theorem}\label{P:10}
Let $\nu\in[0,1)$. The following assertions hold:

$(i)$ The operator ${A_{\nu,b}}$ has equal deficiency indices $n_{\pm}({A_{\nu,b}})=1$;

$(ii)$ $\mathrm{dom}(A_{\nu,b})=\{f\in H^2[0,b]:f(0)=f'(0)=f(b)=0\};$

$(iii)$  $\mathrm{dom}(A^*_{\nu,b})=\{f\in \mathrm{dom}(S^*_{\nu,b}):f(b)=0\}.$
\end{theorem}
\begin{proof}
 It is easily seen that ${S_{\nu,b}}\subset {A_{\nu,b}}\subset {S^*_{\nu,b}}$ and
$\dim(\dom({A_{\nu,b}})/\dom({S_{\nu,b}}))=1$. But, by Proposition~\ref{P:1.3},
$n_\pm({S_{\nu,b}})=2$. Hence, by the Second Neumann formula implies
$n_\pm({A_{\nu,b}})=1$.
\end{proof}
Later on branch of the function $z^\nu$ selected in the plane $\mathbb C$ with a cut
along the positive half--line $\mathbb R_+$ so $z^\nu=x^\nu$ for $z=x>0.$
\begin{proposition}\label{P:11}
Let $\nu\in[0,1)$ and $b<\infty$. Also assume that ${A_{\nu,b}}$  be the Bessel operator
generated by the expression~\eqref{E:1.1} in $L^2(0, b)$  with the domain
\eqref{Dirihle}. Then:
\begin{itemize}
\item[(i)]
Boundary triplet of the operator ${A^*_{\nu,b}}$ can be selected in the form of
\begin{equation}\label{triple}
\mathcal{H}=\mathbb{C}, \quad \Gamma_0^{\nu,b}f=[f, x^{\frac{1}{2}+\nu}]_0, \quad
\Gamma_1^{\nu,b}f=\left\{
\begin{array}{ll}
-(2\nu)^{-1}[f, x^{\frac{1}{2}-\nu}]_0,& \nu\in(0,1),\\
 \ \ [f, x^{\frac{1}{2}}\log (x)]_0,& \nu=0.
\end{array}
\right.
\end{equation}
\item[(ii)] The Weyl function $M_{\nu;b}(\cdot)$ corresponding to the boundary triplet~
\eqref{triple} has the form:
\begin{equation}\label{Weyl_fun}
M_{\nu;b}(z)=\left\{
\begin{array}{ll}
 -\frac{\Gamma(1-\nu)}{2\nu4^{\nu}\Gamma(1+\nu)}\cdot\frac{J_{-\nu}(b\sqrt{z})}{J_{\nu}(b\sqrt{z})}\cdot
z^{\nu},& \nu\in(0,1), \\
 -\log\left(\frac{\sqrt{z}}{2}\right)+\frac{\pi}{2}\frac{Y_0(b\sqrt{z})}{J_0(b\sqrt{z})}-\gamma,&
 \nu=0,
\end{array}
\right.
\end{equation}
where $\gamma$ is Euler's constant.
\end{itemize}
\end{proposition}
\begin{proof}
(i) Let $f,g\in\dom({A^*_{\nu,b}})$. Integrating by parts, we obtain
\begin{eqnarray*}
({A^*_{\nu,b}}f,g)&-&(f,{A^*_{\nu,b}}g)=\lim\limits_{\varepsilon\rightarrow0}\left [
\int\limits_{\varepsilon}^b
\left(-f^{\prime\prime}(x)\overline{g(x)}+\frac{\nu^2-\frac{1}{4}}{x^2}f(x)\right)
\overline{g(x)}dx\right.-\\
&-&\left.\int\limits_{\varepsilon}^b
f(x)\left(\overline{-g^{\prime\prime}(x)}+\frac{\nu^2-\frac{1}{4}}{x^2}\overline{g(x)}\right)dx\right]
=
\lim\limits_{\varepsilon\rightarrow0}\left\{-f(\varepsilon)\overline{g^\prime(\varepsilon)}
+f^\prime(\varepsilon)\overline{g(\varepsilon)}\right\}.
\end{eqnarray*}
On the other hand it is easily seen  that
\[(\Gamma_1^{\nu,b}f,\Gamma_0^{\nu,b}g)-(\Gamma_0^{\nu,b}f,\Gamma_1^{\nu,b}g)=
\]
\[
=\frac{1}{2\nu}\lim\limits_{x\rightarrow
0}\left[\left(\left(\frac{1}{2}+\nu\right)x^{\nu-\frac{1}{2}}f(x)-
x^{\frac{1}{2}+\nu}f^\prime(x)\right)\left(\left(\frac{1}{2}-\nu\right)
x^{-\frac{1}{2}-\nu}\overline{g(x)}-x^{\frac{1}{2}-\nu}\overline{g^\prime(x)}\right)\right.-
\]
\[
\left.-\left(\left(\frac{1}{2}-\nu\right)x^{-\nu-\frac{1}{2}}f(x)-x^{\frac{1}{2}-\nu}f^\prime(x)
\right)\left(\left(\frac{1}{2}+\nu\right)x^{-\frac{1}{2}+\nu}\overline{g(x)}-x^{\frac{1}{2}+\nu}
\overline{g^\prime(x)}\right)\right]
\]
\[
=\frac{1}{2\nu}\lim\limits_{x\rightarrow
0}2\nu(f'(x)\overline{g(x)}-f(x)\overline{g'(x)})=\lim\limits_{x\rightarrow
0}\{-f(x)\overline{g'(x)}+f'(x)\overline{g(x)})\}.
\]
 Comparing this formula with the previous one we obtain the Green's formula
$$({A^*_{\nu,b}}f,g)-(f,{A^*_{\nu,b}}g)=(\Gamma_1^{\nu,b}f,\Gamma_0^{\nu,b}g)
-(\Gamma_0^{\nu,b}f,\Gamma_1^{\nu,b}g).$$

The case $\nu=0$ is considered similarly.

{\rm{(ii.1)}} First we consider the case $\nu\in(0,1)$.

 By the asymptotic relations \eqref{Bessel_0}
$x^{1/2}J_\nu(x\sqrt{z})\in L^2(0,b)$ and $x^{1/2}J_{-\nu}(x\sqrt{z})\in L^2(0,b)$.
Therefore
\begin{equation}\label{E:def_vec}
f_z(x):=x^{\frac{1}{2}}\left(J_{\nu}(x\sqrt{z})-\frac{J_{\nu}(b\sqrt{z})}
{J_{-\nu}(b\sqrt{z})}J_{-\nu}(x\sqrt{z})\right)\in L^2(0,b).
\end{equation}
It is easily seen that $f_z(b)=0$, and hence, $f_z\in\dom({A^*_{\nu,b}})$ and
$({A^*_{\nu,b}}-z)f_z=0$. In other words, deficient space $\mathfrak{N}_z({A_{\nu,b}})$
of the operator ${A_{\nu,b}}$ generated by the vector $f_z$.

Using the asymptotic behavior of the Bessel functions \eqref{Bessel_0} and formula
\eqref{E:reccur+} we obtain
\begin{eqnarray}\label{bessel+}
 \big[x^{1/2}J_{\nu}(x\sqrt{z}), x^{1/2+\nu}\big]_0&=&\lim\limits_{x\rightarrow0}
\left[z^{1/2}x^{1+\nu}J_{\nu+1}(x\sqrt{z})\right]\nonumber\\
&=&\lim\limits_{x\rightarrow0}\left[\frac{z^{1+\nu/2}x^{2(1+\nu)}}{2^{1+\nu}\Gamma(2+\nu)}
(1+O(x^2z))\right]=0,\nonumber\\
\big[x^{1/2}J_{-\nu}(x\sqrt{z}),x^{1/2+\nu}\big]_0
&=&\lim\limits_{x\rightarrow0}\left[-z^{1/2}x^{1+\nu}J_{-\nu-1}(x\sqrt{z})\right]
\\
&=&\lim\limits_{x\rightarrow0}\left[-\frac{z^{-\nu/2}2^{1+\nu}}{\Gamma(-\nu)}(1+O(x^2z))\right]=
-\frac{z^{-\nu/2}2^{1+\nu}}{\Gamma(-\nu)}.\nonumber
\end{eqnarray}
Similarly, using the asymptotic behavior of the Bessel functions \eqref{Bessel_0} and
formula \eqref{E:reccur-} we obtain
\begin{eqnarray}\label{bessel-}
 \big[x^{1/2}J_{\nu}(x\sqrt{z}), x^{1/2-\nu}\big]_0&=&-\lim\limits_{x\rightarrow0}
 \left[z^{1/2}x^{1-\nu}J_{\nu-1}(x\sqrt{z})\right]\nonumber\\
&=&-\lim\limits_{x\rightarrow0}\left[\frac{z^{\nu/2}}{2^{\nu-1}\Gamma(\nu)}(1+O(x^2z))\right]=
-\frac{z^{\nu/2}}{2^{\nu-1}\Gamma(\nu)},\nonumber\\
\big[x^{1/2}J_{-\nu}(x\sqrt{z}),
x^{1/2-\nu}\big]_0&=&-\lim\limits_{x\rightarrow0}\left[z^{1/2}x^{1-\nu}J_{-(\nu-1)}(x\sqrt{z})\right]
\\
&=&-\lim\limits_{x\rightarrow0}\left[\frac{z^{1-\nu/2}2^{\nu-1}x^{2(1-\nu)}}{\Gamma(2-\nu)}(1+O(x^2z))\right]=0.\nonumber
\end{eqnarray}
From the formulas \eqref{triple}, \eqref{E:def_vec} and \eqref{bessel+}, \eqref{bessel-}
we arrive at the relation
\begin{equation}\label{bound}
\Gamma_0^{\nu,b}f_z=\frac{2^{1+\nu}}{\Gamma(-\nu)}\cdot
\frac{J_{\nu}(b\sqrt{z})}{J_{-\nu}(b\sqrt{z})}\cdot z^{-\frac{\nu}{2}}; \qquad
\Gamma_1^{\nu,b}f_z=\frac{1}{\nu2^{\nu}{\Gamma(\nu)}}\cdot
 z^{\frac{\nu}{2}}.
\end{equation}
Hence, relation \eqref{bound} and Definition \ref{D:4.0} yield the fist part of formula
\eqref{Weyl_fun}.

(ii.2) The case $\nu=0$.

 By the asymptotic relations \eqref{Bessel_0}
and \eqref{Bessel_Y} $x^{1/2}J_0(x\sqrt{z})\in L^2(0,b)$ and $x^{1/2}Y_{0}(x\sqrt{z})\in
L^2(0,b)$. Therefore
\begin{equation}\label{E:def_vec_0}
f_z(x):=x^{\frac{1}{2}}\left(J_{0}(x\sqrt{z})-\frac{J_{0}(b\sqrt{z})}
{Y_{0}(b\sqrt{z})}Y_{0}(x\sqrt{z})\right)\in L^2(0,b).
\end{equation}
It is easily seen that $f_z(b)=0$, and, hence, $f_z\in\dom(A^*_{0,b})$ and
$(A^*_{0,b}-z)f_z=0$. In other words, the deficiency space $\mathfrak{N}_z(A_{0,b})$ of
the operator $A_{0,b}$ generated by the vector $f_z$.

Using the asymptotic behavior of the Bessel functions \eqref{Bessel_0} and formula
\eqref{E:reccur+} we obtain
\begin{eqnarray}\label{bessel+_0}
 \big[x^{1/2}J_{0}(x\sqrt{z}), x^{1/2}\big]_0&=&\lim\limits_{x\rightarrow0}
\left[xz^{1/2}J_{1}(x\sqrt{z})\right]\nonumber\\
&=&\lim\limits_{x\rightarrow0}\left[\frac{x^2z}{2}
(1+O(x^2z))\right]=0,\nonumber\\
\big[x^{1/2}Y_{0}(x\sqrt{z}),x^{1/2}\big]_0 &=&\lim\limits_{x\rightarrow0}
\left[xz^{1/2}Y_{1}(x\sqrt{z})\right]\nonumber\\
&=&\lim\limits_{x\rightarrow0}\left[-x\sqrt{z}\cdot\frac{2}{\pi\cdot x\sqrt{z}}(1+O(x^2z))\right]=
-\frac{2}{\pi}.
\end{eqnarray}
Similarly, using the asymptotic behavior of the Bessel functions \eqref{Bessel_0} and
\eqref{Bessel_Y} and formula \eqref{Bessel'} we obtain
\begin{eqnarray}\label{bessel-0}
 \big[x^{1/2}J_{0}(x\sqrt{z}), x^{1/2}\log(x)\big]_0&=&\lim\limits_{x\rightarrow0}
 \left[J_0(x\sqrt{z})+x\log(x)\cdot\sqrt{z}J_1(x\sqrt{z})\right]\nonumber\\
&=&\lim\limits_{x\rightarrow0}\left[\left(1+\frac{x^2\log(x)}{2}z\right)(1+O(x^2z))\right]=
1,\nonumber\\
\big[x^{1/2}Y_{0}(x\sqrt{z}), x^{1/2}\log(x)\big]_0&=&\lim\limits_{x\rightarrow0}
 \left[Y_0(x\sqrt{z})+x\log(x)\cdot\sqrt{z}Y_1(x\sqrt{z})\right]\nonumber\\
&=&\lim\limits_{x\rightarrow0}\left[\frac{2}{\pi}\left[\log\left(\frac{x\sqrt{z}}{2}\right)+\gamma\right]
-\frac{2}{\pi}\log(x)(1+O(x^2z))\right]\nonumber\\
&=&\frac{2}{\pi}\left[\log\left(\frac{\sqrt{z}}{2}\right)+\gamma\right].
\end{eqnarray}
From formulas \eqref{triple}, \eqref{E:def_vec_0},\eqref{bessel+_0} and \eqref{bessel-0}
we arrive at the relation
\begin{equation}\label{bound_0}
\Gamma_0^{0,b}f_z=\frac{2}{\pi}\cdot \frac{J_{0}(b\sqrt{z})}{Y_{0}(b\sqrt{z})}, \qquad
\Gamma_1^{0,b}f_z=1-\frac{2}{\pi}\cdot
\frac{J_{0}(b\sqrt{z})}{Y_{0}(b\sqrt{z})}\left[\log\left(\frac{\sqrt{z}}{2}\right)+\gamma\right].
\end{equation}
Hence, by \eqref{bound_0} and Definition \ref{D:4.0} follows the second part of the
formula \eqref{Weyl_fun}.
\end{proof}
\begin{corollary}
The Weyl function $M_{\nu,b}(z)$ is meromorphic, hence the spectrum of the operator
${A_{\nu,b}}$ is discrete.
\end{corollary}
\begin{proof}
$(i)$ First we consider the case $\nu\in(0,1)$. It follows from the form of Weyl
function~\eqref{Weyl_fun} and the form of the Bessel functions $J_\nu$,
$J_{-\nu}$~\eqref{Jnu} that
\begin{eqnarray*}
M_{\nu,b}(z)&=&-\frac{\Gamma(1-\nu)}{2\nu4^{\nu}\Gamma(1+\nu)}\cdot\frac{J_{-\nu}(b\sqrt{z})}{J_{\nu}(b\sqrt{z})}\cdot
z^{\nu}\\&=&-\frac{\Gamma(1-\nu)}{2\nu
b^{2\nu}\Gamma(1+\nu)}\sum\limits_{k=0}^\infty\frac{\left(-\frac{1}{4}b^2z\right)^k}
{k!\Gamma(-\nu+k+1)}\left\{\sum\limits_{k=0}^\infty\frac{\left(-\frac{1}{4}b^2z\right)^k}
{k!\Gamma(\nu+k+1)}\right\}^{-1}.
\end{eqnarray*}

$(ii)$ The case $\nu=0$. It follows from the form of Weyl function~\eqref{Weyl_fun} and
the form of the Bessel functions $J_0$, $Y_0$~\eqref{J0},~\eqref{Y0} that
$$
M_{0,b}(z)=-\log\left(\frac{\sqrt{z}}{2}\right)+\frac{\pi}{2}\frac{Y_0(b\sqrt{z})}{J_0(b\sqrt{z})}-\gamma=
\log (b)+\left\{
\frac{\frac{1}{4}b^2z}{(1!)^2}-\left(1+\frac{1}{2}\right)\frac{\left(\frac{1}{4}b^2z\right)^2}{(2!)^2}+
\left(1+\frac{1}{2}+\frac{1}{3}\right)\frac{\left(\frac{1}{4}b^2z\right)^3}{(3!)^2}-\ldots\right\}$$
$$
\times\left\{1-\frac{\frac{1}{4}b^2z}{(1!)^2}+\frac{\left(\frac{1}{4}b^2z\right)^2}{(2!)^2}-
\frac{\left(\frac{1}{4}b^2z\right)^3}{(3!)^2}+\ldots\right\}^{-1}.
$$
\end{proof}
\begin{proposition}\label{Friendrichs.}
Let $\nu\in[0,1)$ and $\Pi_{\nu,b}=\{\mathcal{H}, \Gamma_0^{\nu,b}, \Gamma_1^{\nu,b}\}$
be the boundary triplet of the form~\eqref{triple} for the operator ${A^*_{\nu,b}}$.
Then:
\begin{itemize}
\item[\textit{(i)}] The domain of the Friedrichs extension ${A_{\nu,b}}_{F}$
of the operator ${A_{\nu,b}}$ has the form
\begin{eqnarray}\label{Fri_f}
\dom({A_{\nu,b}}_{F})=\ker(\Gamma_0^{\nu,b})=\left\{f\in \dom({A^*_{\nu,b}}):[f,
x^{\frac{1}{2}+\nu}]_0=0\right\}.
\end{eqnarray}

\item[\textit{(ii)}] The domain of the Krein
extension ${A_{\nu,b}}_{K}$ of the operator ${A_{\nu,b}}$ has the form
\begin{equation}\label{Kre_f}
\dom({A_{\nu,b}}_{K})=\left\{
\begin{array}{ll}
\left\{f\in\dom({A^*_{\nu,b}}):(2\nu)^{-1}[f,b^{-2\nu}x^{1/2+\nu}-x^{1/2-\nu}]_0=0\right\},
&
\nu\in(0,1),\\
\left\{f\in \dom(A^*_{0,b}):[f,x^\frac{1}{2}\log\left(\frac{x}{b}\right)]_0=0\right\},&
\nu=0.
\end{array}
\right.
\end{equation}
\end{itemize}
\end{proposition}
\begin{proof}
 (i) First we consider the case $\nu\in(0,1)$.

 Applying the asymptotic behavior of the Bessel functions \eqref{asym_infty} to the
 Weyl function~\eqref{Weyl_fun}, we obtain
\begin{eqnarray*}
 M_{\nu,b}(-\infty)&\cdot&\frac{2\nu4^{\nu}\Gamma(1+\nu)}{\Gamma(1-\nu)}=\frac{2\nu4^{\nu}\Gamma(1+\nu)}{\Gamma(1-\nu)}
 \lim\limits_{z\rightarrow-\infty}M_{\nu,b}(z)=
 -\lim\limits_{z\rightarrow-\infty}
 \cdot\frac{J_{-\nu}(b\sqrt{z})}{J_{\nu}(b\sqrt{z})}\cdot
z^{\nu}
\\
&=&-\lim\limits_{x\rightarrow+\infty}
 \frac{J_{-\nu}(ib\sqrt{-x})}{J_{\nu}(ib\sqrt{-x})}\cdot
(-x)^{\nu} =-\lim\limits_{x\rightarrow+\infty}
\left[\frac{\cos\left(ib\sqrt{-x}+\frac{\nu\pi}{2}-\frac{\pi}{4}\right)}
{\cos\left(ib\sqrt{-x}-\frac{\nu\pi}{2}-\frac{\pi}{4}\right)}\cdot{(-x)^{\nu}}\right]
\\
&=& -\lim\limits_{x\rightarrow+\infty}
\left[(-x)^{\nu}\cdot\frac{e^{-i(ib\sqrt{-x}+\frac{\nu\pi}{2}-\frac{\pi}{4})}+o(1)}
{e^{-i(ib\sqrt{-x}-\frac{\nu\pi}{2}-\frac{\pi}{4})}+o(1)}\right]
=-e^{-i\nu\pi}\lim\limits_{x\rightarrow+\infty}(-x)^{\nu} \\
&=& -\frac {e^{i\nu\pi}}{e^{i\nu\pi}}\lim\limits_{x\rightarrow+\infty} x^{\nu}
=-\lim\limits_{x\rightarrow+\infty} x^{\nu}=-\infty.
\end{eqnarray*}

 The case $\nu=0$.

Applying the asymptotic behavior of the Bessel functions \eqref{asym_infty} to the
 Weyl function~\eqref{Weyl_fun}, we obtain
\begin{eqnarray*}
M_{0,b}(-\infty)&=&\lim\limits_{z\rightarrow-\infty}M_{0,b}(z)=\lim\limits_{z\rightarrow-\infty}\left[
-\log\left(\frac{\sqrt{z}}{2}\right)+\frac{\pi}{2}\frac{Y_0(b\sqrt{z})}{J_0(b\sqrt{z})}-\gamma\right]\\
&=&
\lim\limits_{x\rightarrow+\infty}\left[-\log\left(\frac{i\sqrt{x}}{2}\right)+\frac{\pi}{2}\frac{Y_0(ib\sqrt{x})}{J_0(ib\sqrt{x})}-\gamma\right]\\
&=&\lim\limits_{x\rightarrow+\infty}\left[-\frac{\pi}{2}i-\log(\sqrt{x})+\frac{\pi}{2}\cdot
\frac{\sin(bi\sqrt{x}-\frac{\pi}{4})}{\cos(bi\sqrt{x}-\frac{\pi}{4})}-\gamma\right]\\
&=&\lim\limits_{x\rightarrow+\infty}
\left[-\frac{\pi}{2}i-\log(\sqrt{x})+\frac{\pi}{2}\cdot i-\gamma\right]=-\infty
\end{eqnarray*}
So, by Proposition~\ref{prkf}, relation \eqref{Fri_f} is valid.

(ii.1) First we consider the case $\nu\in(0,1)$.

 Applying the asymptotic behavior of the Bessel functions~\eqref{Bessel_0} to the
 Weyl function~\eqref{Weyl_fun}, we obtain
\begin{eqnarray}\label{M0_ab_nu}
M_{\nu,b}(0)&=&\lim\limits_{z\rightarrow-0}M_{\nu,b}(z)=\lim\limits_{z\rightarrow-0}\left[
-\frac{\Gamma(1-\nu)}{2\nu4^{\nu}\Gamma(1+\nu)}\cdot\frac{J_{-\nu}(b\sqrt{z})}{J_{\nu}(b\sqrt{z})}\cdot
z^{\nu} \right]\nonumber\\
&=&-\lim\limits_{z\rightarrow-0}
\left[\frac{\Gamma(1-\nu)}{2\nu\Gamma(1+\nu)4^{\nu}}\cdot\frac{\Gamma(1+\nu)4^{\nu}}{\Gamma(1-\nu)}
 \cdot b^{-2\nu}z^{-\nu}z^{\nu}\right]=-\frac{b^{-2\nu}}{2\nu}.
 \end{eqnarray}
 The first part of relation \eqref{Kre_f} follows from Proposition~\ref{prkf}.

(ii.2) The case $\nu=0$.
\begin{eqnarray}\label{M0_ab_0}
M_{0,b}(0)&=&\lim\limits_{z\rightarrow-0}M_{0,b}(z)=\lim\limits_{z\rightarrow-0}\left[
-\log\left(\frac{\sqrt{z}}{2}\right)+\frac{\pi}{2}\frac{Y_0(b\sqrt{z})}{J_0(b\sqrt{z})}-\gamma\right]\nonumber\\
&=&\lim\limits_{z\rightarrow-0}
\left[-\log\left(\frac{\sqrt{z}}{2}\right)+\frac{\pi}{2}\cdot\frac{2}{\pi}
\left(\log\left(\frac{b\sqrt{z}}{2}\right)+\gamma\right)-\gamma\right]= \log(b).
 \end{eqnarray}
The second part of relation~\eqref{Kre_f} follows from Proposition~\ref{prkf}.
\end{proof}

\begin{remark}
We note that by virtue of the formulas~\eqref{Fri_f} and~\eqref{Kre_f} domain of the
Friedrichs extension does not depend on $b$, and the Krein extension depends.
\end{remark}
\begin{corollary}\label{C:Ext}

{\it{(i)}} 
%
For $\nu\in(0,1)$ extension  ${ A}_{{\nu,b}_h}$ is  non-negative, ${{
A}_{{\nu,b}_h}}\geqslant 0$ if and only if
\[h\geqslant -\frac{b^{-2\nu}}{2\nu}.
\]

{\it{(ii)}} For $\nu=0$ extension ${{ A}_{{0,b}_h}}$ is  non-negative, ${{
A}_{{0,b}_h}}\geqslant 0$ if and only if
\[h\geqslant \log(b).
\]
\end{corollary}
\begin{proof}
(i)  By virtue of the Proposition~\ref{prkf}~(ii), ${{ A}_{{\nu,b}_0}}$ is the
Friedrichs extension. From~\eqref{M0_ab_nu} it  follows that
$M_{\nu,b}(0)=-\frac{b^{-2\nu}}{2\nu}$ and
   then, by virtue of the Proposition~\ref{prkf}~(iii), the
   extension
   ${{ A}_{{\nu,b}_h}}$ is a non-negative,
   ${{ A}_{{\nu,b}_h}}\geqslant 0$
if and only if $h\geqslant M_{\nu,b}(0)=-\frac{b^{-2\nu}}{2\nu}.$

(ii) By virtue of the Proposition~\ref{prkf}~(ii), ${{ A}_{{0,b}_0}}$ is the Friedrichs
extension. From~\eqref{M0_ab_0} it  follows that $M_{0,b}(0)=\log(b)$ and
    then, by virtue of the Proposition~\ref{prkf}~(iii), the
   extension
   ${{ A}_{{0,b}_h}}$ is a non-negative,
   ${{ A}_{{0,b}_h}}\geqslant 0$
if and only if $h\geqslant M_{0,b}(0)=\log(b).$
\end{proof}

\begin{remark}\label{prin_s}
Note that, for $\nu\in(0,1)$, the solution $x^{1/2+\nu}\in\dom({A_{\nu,b}}_{F})$, while
the solution $x^{1/2-\nu}\not\in\dom({A_{\nu,b}}_{F})$. So $x^{1/2+\nu}$ is the
principal solution at $0$ (see \cite[Def. 11.5]{EGNT}). Similarly, for $\nu=0$, the
solution $x^{1/2}$ is the principal solution at $0$, while $x^{1/2}\log (x)$ is not.
\end{remark}

Indeed,
$$
[x^{1/2+\nu},x^{1/2-\nu}]_0=\lim\limits_{x\to0}\left\{\left(\frac{1}{2}-\nu\right)
x^{1/2+\nu}x^{-1/2-\nu}-\left(\frac{1}{2}+\nu\right)x^{1/2-\nu}x^{-1/2+\nu}
\right\}=-2\nu\neq0.
$$
Therefore, by Proposition \ref{Friendrichs.}, $x^{1/2-\nu}\not\in\dom({A_{\nu,b}}_{F})$.

The case $\nu=0$ is considered similarly.

\section{The Bessel  operator ${A_{\nu,\infty}}$ on half--line}

Here, we consider the minimal Bessel operator ${A_{\nu,\infty}}$ generated by the
expression~\eqref{E:1.1} in $L^2(\mathbb R_+)$ for $\nu\in[0,1)$.

Let $D_{\min}^2$ be a minimal differential operator of the 2nd order, generated in
$L^2(\mathbb R_+)$ by differential expression $-d^2/dx^2,$
     \begin{equation}\label{E:1.127}
\dom(D_{\min}^2)=H^2_0(\mathbb R_+)=\{f \in H^2(\mathbb R_+): f(0)=f'(0)=0\}.
\end{equation}

\begin{theorem}\label{P:10}
Let $\nu\in[0,1)$.  Then the following assertions hold:

$(i)$ The operator ${A_{\nu,\infty}}$ has equal deficiency indices
$n_{\pm}({A_{\nu,\infty}})=1.$

$(ii)$ The domain of the operator ${A_{\nu,\infty}}$ is given by
\begin{equation}\label{Dom}
\dom({A_{\nu,\infty}})=H_0^2(\mathbb R_+).
\end{equation}

$(iii)$ ${A_{\nu,\infty}}_{\max}={A^*_{\nu,\infty}}$ and
\begin{equation}\label{dom*}
\dom({A^*_{\nu,\infty}})= \left\{
\begin{array}{ll}
{H}_0^2(\mathbb R_+)\dot{+}\spann\{x^{1/2+\nu}\xi(x),
x^{1/2-\nu}\xi(x)\},& \nu\in(0,1),\\
{H}_0^2(\mathbb R_+)\dotplus \spann \{x^{1/2}\xi(x), x^{1/2}\log(x)\xi(x)\},& \nu=0,\\
\end{array}
\right.
\end{equation}
where $\xi\in C_0^1(\mathbb R_+)$ is a some function such that $\xi(x)=1$
 for $x\in[0,1]$.
\end{theorem}
\begin{proof}
{(i)--(ii)} 
The function $u \in \widetilde{H}_0^{2}(\mathbb R_+)$  admits the integral
representation $u(x)=\int \limits_0^{x}(x-t) u''(t)dt.$ Therefore,
\begin{equation}\label{E:1.11i}
Q u(x)=\frac{1}{x^2}u(x)=\frac{1}{x^2}\int \limits_0^{x}(x-t) u''(t)dt=(Q\mathcal
I^2(D_{\min}^2u))(x).
\end{equation}
 By virtue of Lemma~\ref{L:1.2}, this yields
\begin{eqnarray}\label{E:300.43i}
  \|Q u\|_{2}&=&\left\|\frac{1}{x^2}u\right\|_{2}=\left\|Q\mathcal I^2D_{\min}^2 u\right\|_{2}
  \leqslant \|Q\mathcal I^2\|_2 \cdot\|D_{\min}^2 u\|_{2} \nonumber\\
&=&\frac{4}{3}\|D_{\min}^2u\|_{2} \leqslant\frac{4}{3}\|u\|_{H^2_0(\mathbb R_+)}.
     \end{eqnarray}
It is easy to see that $\nu^2-\frac{1}{4}$ admits the representation
$\nu^2-\frac{1}{4}=\frac{3}{4}(1-\varepsilon)$, where $\varepsilon >0.$ Then
relation~\eqref{E:300.43i} implies the estimate
\begin{equation}\label{E:300.44i}
\left\|\left(\nu^2-\frac{1}{4}\right)
  Qu\right\|_{2}=\left|\nu^2-\frac{1}{4}\right|\cdot\|Qu\|_{2}\leqslant\frac{3}{4}(1-\varepsilon)
  \cdot\frac{4}{3}\|u\|_{H^2_0[0,b]}
  =(1-\varepsilon)\|u\|_{H^2_0[0,b]}, \quad u\in H^2_0[0,b].
\end{equation}
 Estimate~\eqref{E:300.44i} means that $Q$ is strongly $D_{\min}^2$-bounded. Therefore,
 by the Kato--Rellich theorem (see \cite{Ka})
$n_{\pm}({A_{\nu,\infty}})=n_{\pm}(D_{\min}^2)=1$ and
$\dom({A_{\nu,\infty}})=H_0^{2}(\mathbb R_+)$.

{{(iii)}} Since
$$
\tau_\nu x^{1/2\pm\nu}\xi(x)=0,
$$
where the equality is understood in the meaning of the theory of distributions, and
$x^{1/2\pm\nu}\xi(x) \in L^2(\mathbb R_+)$, then
$$\{x^{1/2+\nu}\xi(x), x^{1/2-\nu}\xi(x)\} \subset\dom({A_{\nu,\infty}}_{\max})=\dom({A^*_{\nu,\infty}}),$$
and  $\ker({A^*_{\nu,\infty}})=\{x^{1/2+\nu}\xi(x), x^{1/2-\nu}\xi(x)\}$ $\subset
L^2(\mathbb R_+)$. In addition, it is clear that ${H}_0^2(\mathbb R_+) \subset
\dom({A^*_{\nu,\infty}})$ and $\dim({H}_0^2(\mathbb R_+))/$ $\dom({A_{\nu,\infty}}))=2.$
On the other hand, since $n_{\pm}({A_{\nu,\infty}})=1,$ we have
$\dim(\dom({A^*_{\nu,\infty}})/\dom({A_{\nu,\infty}})) = 2n_{\pm}({A_{\nu,\infty}})=2$
by the first Neumann formula. Therefore, formula~\eqref{dom*} is valid.

The case $\nu=0$ is considered similarly.
\end{proof}

\begin{remark}
 In~\cite[Proposition 4.11]{BDG} proved that for $0<\mathrm{Re}\ \nu<1$ for
$f\in\mathrm{dom}({A_{\nu,\infty}})$ the relations $  f(x)=o(x^{3/2}),\
f'(x)=o(x^{1/2})$  are valid for $x\to0$, and for $\nu=0$ the relations $
f(x)=o(x^{3/2}\log(x)),\  f'(x)=o(x^{1/2}\log(x))$ are valid for $x\to0$, which are
easily follow from~\eqref{Dom}.
\end{remark}

Next we compute the Weyl function and the corresponding spectral function of the
operator ${A_{\nu,\infty}}$ using the boundary triplet technique.
\begin{proposition}\label{P:110}
Let $\nu\in[0,1)$. Then:
\begin{itemize}
\item[(i)]
The boundary triplet of the operator ${A^*_{\nu,\infty}}$ can be selected in the form:
\begin{equation}\label{triple1}
\mathcal{H}=\mathbb{C}, \quad \Gamma_0^{\nu,\infty}f=[f, x^{\frac{1}{2}+\nu}]_0, \quad
\Gamma_1^{\nu,\infty}f=\left\{
\begin{array}{ll}
  -(2\nu)^{-1}[f, x^{\frac{1}{2}-\nu}]_0, & \nu\in(0,1),\\
  \ \ [f, x^{\frac{1}{2}}\log (x)]_0, & \nu=0.
\end{array}
\right.
\end{equation}
\item[(ii)] The corresponding Weyl function $M_{\nu;\infty}(\cdot)$
has the form:
\begin{equation}\label{Weyl_fun1}
M_{\nu;\infty}(z)= \left\{
\begin{array}{l}
  e^{i(1-\nu)\pi}\frac{\Gamma(1-\nu)}{2\nu4^{\nu}\Gamma(1+\nu)}
  z^{\nu},
  \quad \nu\in(0,1), \\
-\log\left(\frac{\sqrt{z}}{2}\right)+\frac{i\pi}{2}-\gamma,
  \quad\nu=0,
\end{array}
\right.\quad z\in\mathbb C\setminus\mathbb R_+,
\end{equation}
where $\gamma$ is Euler's constant.
\item[(iii)] The spectral function $\Sigma_\nu(t)$ of the operator ${A_{\nu,\infty}}_0
={A^*_{\nu,\infty}}\upharpoonright \ker\Gamma_0^{\nu,\infty}$ is given by
\begin{equation}\label{Spec_fun1}
\Sigma_\nu(t) = \frac{t^{\nu+1}}{ 2^{2\nu+1}\Gamma^2(1+\nu)}\chi_{[0,\infty)}(t).
\end{equation}
\end{itemize}
\end{proposition}
\begin{proof}
(i) Let $f,g\in\dom({A^*_{\nu,\infty}})$. Integrating by parts, we obtain
\[({A^*_{\nu,\infty}}f,g)-(f,{A^*_{\nu,\infty}}g)=\lim\limits_{\varepsilon\rightarrow0}\left(
\int\limits_{\varepsilon}^\infty
\left(-f^{\prime\prime}(x)\overline{g(x)}+\frac{\nu^2-\frac{1}{4}}{x^2}f(x)\right)
\overline{g(x)}dx\right.-
\]
\[
-\left.\int\limits_{\varepsilon}^\infty
f(x)\left(\overline{-g^{\prime\prime}(x)}+\frac{\nu^2-\frac{1}{4}}{x^2}\overline{g(x)}\right)dx\right]
=
\lim\limits_{\varepsilon\rightarrow0}\left\{-f(\varepsilon)\overline{g^\prime(\varepsilon)}
+f^\prime(\varepsilon)\overline{g(\varepsilon)}\right\}.
\]

On the other hand, it is easily seen  that
\[(\Gamma_1^{\nu,\infty}f,\Gamma_0^{\nu,\infty}g)-(\Gamma_0^{\nu,\infty}f,\Gamma_1^{\nu,\infty}g)=
\]
\[
=\frac{1}{2\nu}\lim\limits_{x\rightarrow
0}\left[\left(\left(\frac{1}{2}+\nu\right)x^{\nu-\frac{1}{2}}f(x)-
x^{\frac{1}{2}+\nu}f^\prime(x)\right)\left(\left(\frac{1}{2}-\nu\right)
x^{-\frac{1}{2}-\nu}\overline{g(x)}-x^{\frac{1}{2}-\nu}\overline{g^\prime(x)}\right)\right.-
\]
\[
\left.-\left(\left(\frac{1}{2}-\nu\right)x^{-\nu-\frac{1}{2}}f(x)-x^{\frac{1}{2}-\nu}f^\prime(x)
\right)\left(\left(\frac{1}{2}+\nu\right)x^{-\frac{1}{2}+\nu}\overline{g(x)}-x^{\frac{1}{2}+\nu}
\overline{g^\prime(x)}\right)\right]
\]
\[
=\frac{1}{2\nu}\lim\limits_{x\rightarrow
0}2\nu(f'(x)\overline{g(x)}-f(x)\overline{g'(x)})=\lim\limits_{x\rightarrow
0}\left\{-f(x)\overline{g'(x)}+f'(x)\overline{g(x)})\right\}.
\]
Comparing this formula with the previous one, we obtain the Green's formula
$$({A^*_{\nu,\infty}}f,g)-(f,{A^*_{\nu,\infty}}g)=
(\Gamma_1^{\nu,\infty}f,\Gamma_0^{\nu,\infty}g)-(\Gamma_0^{\nu,\infty}f,\Gamma_1^{\nu,\infty}g).$$

The case $\nu=0$ is considered similarly.

(ii.1) First, we consider the case $\nu\in(0,1)$.

By the asymptotic relations \eqref{Bessel_0} and \eqref{Bessel_Y},
$x^{1/2}J_\nu(x\sqrt{z})\in L^2(\mathbb R_+)$ and $x^{1/2}Y_{\nu}(x\sqrt{z})\in
L^2(\mathbb R_+)$. Therefore
\begin{equation}\label{E:def_vec1}
f_z(x)=x^{\frac{1}{2}}\left\{J_{\nu}(x\sqrt{z})+iY_{\nu}(x\sqrt{z})\right\}\in
L^2(\mathbb R_+).
\end{equation}
It is easily seen  that $\lim\limits_{x\rightarrow\infty}f_z(x)=0$. So
$f_z\in\dom({A^*_{\nu,\infty}})$ and $({A^*_{\nu,\infty}}-z)f_z=0$. In other words, the
deficiency space $\mathfrak{N}_z({A_{\nu,\infty}})$ of the operator ${A_{\nu,\infty}}$
generated by the vector $f_z$.

Using the asymptotic behavior of the Bessel functions~\eqref{Bessel_0} and
formula~\eqref{E:reccur+}, we obtain
\begin{eqnarray}\label{besselY+}
 \left[x^{1/2}Y_{\nu}(x\sqrt{z}), x^{1/2+\nu}\right]_0&=&
 \left[x^{1/2}\frac{J_\nu(x\sqrt{z})\cos(\nu\pi)-J_{-\nu}(x\sqrt{z})}
 {\sin(\nu\pi)},
 x^{1/2+\nu}\right]_0\nonumber\\
 &=&-\frac{\nu2^{1+\nu}}{\sin(\nu\pi)\Gamma(1-\nu)} z^{-\nu/2}.
 \end{eqnarray}
 Similarly, using the asymptotic behavior of the Bessel functions~\eqref{Bessel_0} and
formula~\eqref{E:reccur-}, we obtain
\begin{eqnarray}\label{besselY-}
 \left[x^{1/2}Y_{\nu}(x\sqrt{z}), x^{1/2-\nu}\right]_0&=&
 \left[x^{1/2}\frac{J_\nu(x\sqrt{z})\cos(\nu\pi)-J_{-\nu}(x\sqrt{z})}
 {\sin(\nu\pi)},
 x^{1/2-\nu}\right]_0\nonumber\\
 &=&-\frac{\nu\cos(\nu\pi)}{\sin(\nu\pi)2^{\nu-1}\Gamma(1+\nu)} z^{\nu/2}.
\end{eqnarray}
From the
formulas~\eqref{bessel+},~\eqref{bessel-},~\eqref{triple1},~\eqref{E:def_vec1},~\eqref{besselY+}
and~\eqref{besselY-}, we arrive at the relation
\begin{eqnarray}\label{bound1}
\Gamma_0^{\nu,\infty}f_z=-\frac{i\nu2^{\nu+1}}{\sin(\nu\pi)\Gamma(1-\nu)} z^{-\nu/2};
\end{eqnarray}
\begin{eqnarray}\label{bound11}
\Gamma_1^{\nu,\infty}f_z=\left(1+i\frac{\cos(\nu\pi)}{\sin(\nu\pi)}\right)
\frac{z^{\frac{\nu}{2}}}{2^{\nu}\Gamma(1+\nu)}=
\frac{e^{i\pi(1-\nu)}}{i\sin(\nu\pi)}\cdot\frac{z^{\frac{\nu}{2}}}{2^{\nu}\Gamma(1+\nu)}.
\end{eqnarray}
Hence, by~\eqref{bound1},~\eqref{bound11}, and Definition~\ref{D:4.0}, we obtain the
first part of the formula \eqref{Weyl_fun1}.

(ii.2) The case $\nu=0$.

By the asymptotic relations~\eqref{Bessel_0} and~\eqref{Bessel_Y},
$x^{1/2}J_0(x\sqrt{z})\in L^2(\mathbb R_+)$ and $x^{1/2}Y_{0}(x\sqrt{z})\in L^2(\mathbb
R_+)$. Therefore
\begin{equation}\label{E:def_vec1_0}
f_z(x)=x^{\frac{1}{2}}\left\{J_{0}(x\sqrt{z})+iY_{0}(x\sqrt{z})\right\}\in
L^2(\mathbb R_+).
\end{equation}
It is easily seen  that $\lim\limits_{x\rightarrow\infty}f_z(x)=0$. So
$f_z\in\dom({A^*_{0,\infty}})$ and $({A^*_{0,\infty}}-z)f_z=0$. In other words, the
deficiency space $\mathfrak{N}_z({A_{0,\infty}})$ of the operator ${A_{0,\infty}}$
generated by the vector $f_z$.

From formulas~\eqref{bessel+_0},~\eqref{bessel-0},~\eqref{triple1}
and~\eqref{E:def_vec1_0},  we arrive at the relations
\begin{eqnarray}\label{bound1_0}
\Gamma_0^{0,\infty}f_z=-\frac{2}{\pi} i;
\end{eqnarray}
\begin{eqnarray}\label{bound11_0}
\Gamma_1^{0,\infty}f_z=1+\frac{2i}{\pi}\left[\log\left(\frac{\sqrt{z}}{2}\right)+\gamma\right].
\end{eqnarray}
Hence, by~\eqref{bound1_0},~\eqref{bound11_0}, and Definition~\ref{D:4.0}, we get the
second part of the formula~\eqref{Weyl_fun1}.

(iii) Since $M_{\nu,\infty}(t+iy)$ is bounded in the rectangle
$(0,\infty)\times(0,y_0)$, its representing measure is absolutely continuous. By Fatou's
Theorem
for $\nu\in(0,1)$
\begin{eqnarray*}
\Sigma_\nu^{'}(t)&=&\frac{1}{\pi}\imm
M_{\nu,\infty}(t+i0)=\frac{1}{\pi}\frac{\Gamma(1-\nu)}{2\nu4^{\nu}\Gamma(1+\nu)}
\imm\left(e^{i(1-\nu)\pi}t^\nu\right)\\
&=&\frac{1}{\pi}\frac{\Gamma(1-\nu)}{2\nu4^{\nu}\Gamma(1+\nu)}t^\nu\imm(e^{i(1-\nu)})
=\frac{(\nu+1)t^{\nu}}{ 2^{2\nu+1}\Gamma^2(1+\nu)}.
\end{eqnarray*}

The case $\nu=0$ is considered similarly.
\end{proof}

\begin{remark}\label{int_r}
In addition, for $\nu\in(0,1)$, the Weyl function $M_{\nu,\infty}(\cdot)$ admits the
integral representation
\begin{equation}\label{Spec_fun}
M_{\nu,\infty}(z)=A_\nu+\frac{1}{
2^{2\nu+1}\Gamma^2(1+\nu)}\int\limits_{-\infty}^\infty\left(\frac{1}{t-z}-\frac{t}{1+t^2}\right)t^\nu
dt,
\end{equation}
where
\begin{eqnarray*}
A_\nu
&=&-\frac{\Gamma(1-\nu)}{2\nu4^{\nu}\Gamma(1+\nu)}\cos\left(\frac{\nu\pi}{2}\right).
\end{eqnarray*}

Similarly, for $\nu=0$, the Weyl function $M_{0,\infty}(\cdot)$ admits the integral
representation
\begin{equation}\label{Spec_funn}
M_{0,\infty}(z)=A_0+\frac{1}{2}\int\limits_{-\infty}^\infty\left(\frac{1}{t-z}-\frac{t}{1+t^2}\right)
dt,
\end{equation}
where the constant
\begin{equation*}
A_0=-\frac{\pi}{4}-\gamma+\log(2).
\end{equation*}
\end{remark}
\begin{remark}
For the Bessel operators,  the formulas similar to \eqref{triple1}  have been obtained
in~\cite[Theorem 2]{Kochybei} for $\nu\in(0,1/2)\cup(1/2,1)$  and in~\cite[Theorem
3]{Kochybei} for $\nu=0$.
\end{remark}

\begin{proposition}\label{Friendrichs}
Let $\nu\in[0,1)$ and $\Pi_{\nu;\infty}=\{\mathcal{H}, \Gamma_0^{\nu,\infty},
\Gamma_1^{\nu,\infty}\}$ be a boundary triplet for the operator ${A^*_{\nu,\infty}}$ of
the form~\eqref{triple1}. Then:
\begin{itemize}
\item[\textit{(i)}] The domain of the Friedrichs extension ${A_{\nu,\infty}}_{F}$
of the operator ${A_{\nu,\infty}}$ has the form
\begin{eqnarray}\label{Fri_fi}
\dom({A_{\nu,\infty}}_{F})=\ker(\Gamma_0^{\nu,\infty})=\left\{f\in
\dom({A^*_{\nu,\infty}}) :[f, x^{\frac{1}{2}+\nu}]_0=0\right\}.\ \
\end{eqnarray}
\item[\textit{(ii)}] The domain of the Krein extension ${A_{\nu,\infty}}_{K}$
of the operator ${A_{\nu,\infty}}$ has the form
\begin{equation}\label{Krein_i}
\dom({A_{\nu,\infty}}_{K})=\left\{
\begin{array}{ll}
\{f\in \dom({A^*_{\nu,\infty}}):[f, x^{\frac{1}{2}-\nu}]_0=0\}, &
\nu\in(0,1),\\
\{f\in \dom({A^*_{0,\infty}}) :[f, x^{\frac{1}{2}}]_0=0\}=\ker(\Gamma_0^{0;\infty}),
&\nu=0.
\end{array}
\right.
\end{equation}
\end{itemize}
In particular, in the case of $\nu=0$ the Friedrichs and Krein extensions coincide
$${A_{0,\infty}}_{F}={A_{0,\infty}}_{K}.$$
\end{proposition}
\begin{proof}
To prove these statements, we use \cite{DM_91}.

 (i.1) For $\nu\in(0,1)$,
\begin{eqnarray*}
M_{\nu,\infty}(-\infty)&=&\lim\limits_{z\rightarrow-\infty}M_{\nu,\infty}(z)=
\lim\limits_{x\rightarrow\infty}
\left[e^{i(1-\nu)\pi}\frac{\Gamma(1-\nu)}{2\nu4^{\nu}\Gamma(1+\nu)}\cdot{(-x)^{\nu}}
\right]\\
&=&-\lim\limits_{x\rightarrow\infty}
\left[\frac{1}{(-1)^\nu}\cdot\frac{\Gamma(1-\nu)}{2\nu4^{\nu}\Gamma(1+\nu)}\cdot{(-1)^{\nu}x^\nu}\right]
=-\infty.
\end{eqnarray*}
By Proposition~\ref{prkf}, the first part of relation~\eqref{Fri_fi} is valid.

(i.2) For $\nu=0$,
\begin{eqnarray}\label{Weyl_-inf0}
M_{0,\infty}(-\infty)&=&\lim\limits_{z\rightarrow-\infty}M_{0,\infty}(z)=\lim\limits_{z\rightarrow-\infty}
\left[\frac{i\pi}{2}-\gamma-\log\left(\frac{\sqrt{z}}{2}\right)\right]\nonumber\\
&=&\lim\limits_{x\rightarrow\infty}
\left[\frac{i\pi}{2}-\gamma-\log\left(i\frac{\sqrt{x}}{2}\right)\right]=
\lim\limits_{x\rightarrow\infty}
\left[-\gamma-\log\left(\frac{\sqrt{x}}{2}\right)\right]=-\infty.
\end{eqnarray}

By Proposition~\ref{prkf}, the second part of relation~\eqref{Fri_fi} is valid.

 (ii.1) First, we consider the case $\nu\in(0,1)$:
\begin{eqnarray}\label{M0_inf_nu}
M_{\nu,\infty}(0)&=&\lim\limits_{z\rightarrow-0}M_{\nu,\infty}(z)=-\lim\limits_{z\rightarrow-0}
\left[e^{i(1-\nu)\pi}\frac{\Gamma(1-\nu)}{2\nu4^{\nu}\Gamma(1+\nu)}\cdot{z^{\nu}}\right]\\
&=&-\lim\limits_{z\rightarrow-0}
\left[\frac{1}{(-1)^\nu}\cdot\frac{\Gamma(1-\nu)}{2\nu4^{\nu}\Gamma(1+\nu)}\cdot{z^\nu}\right]
=0.
\end{eqnarray}
By Proposition~\ref{prkf}, the first part of relation~\eqref{Krein_i} is valid.

(ii.2) The case $\nu=0$:
\begin{eqnarray}\label{Weyl_0}
M_{0,\infty}(0)&=&\lim\limits_{z\rightarrow-0}M_{0,\infty}(z)=\lim\limits_{z\rightarrow-0}
\left[\frac{i\pi}{2}-\gamma-\log\left(\frac{\sqrt{z}}{2}\right)\right]=+\infty.
\end{eqnarray}
By Proposition~\ref{prkf}, the second part of relation~\eqref{Krein_i} is valid.
\end{proof}

\begin{remark}\label{r_spec_f}
In~\cite{EvKalf} the formulas~\eqref{Weyl_fun1},~\eqref{Spec_fun1} and~\eqref{Fri_fi}
were obtained by W.N. Everitt and H. Kalf by using the classical definitions of the
Weyl--Titchmarsch function and Friedrichs extension. The
formulas~\eqref{Fri_fi},~\eqref{Krein_i} can be found for example in~\cite{BDG}.
However, we emphasize that their association with our formula~\eqref{dom*} gives an
explicit description of Friedrichs and Krein extensions
\end{remark}

%

\begin{corollary}  Let $\nu\in[0,1)$ and
\begin{equation}\label{d_e_a}
A_{\nu_a,\infty}=-y''(x)+\left(\frac{\nu^2-\frac{1}{4}}{x^2}-\frac{a}{x}\right)y(x)
\end{equation}
be the operator on the half--line $\mathbb R_+$, $a>0$. Then following assertions hold:

$(i)$ The domain of the operator $A_{\nu_a,\infty}$ coincides with the
domain~\eqref{Dom} of $A_{\nu,\infty}$.

$(ii)$ The domain of the operator $A_{\nu_a,\infty}^*$ coincides with the
domain~\eqref{dom*} of $A_{\nu,\infty}^*$.

$(iii)$ The Friedrichs and Krein extensions of the operator $A_{\nu_a,\infty}^*$
coincides with the Friedrichs and Krein extensions of $A_{\nu,\infty}^*$.
\end{corollary}
\begin{proof}
Since perturbations embedded
$$
\left\|\frac{a}{x}f\right\|_{L^2[0,1]}\leqslant\left(\nu^2-\frac{1}{4}\right)\left\|
\frac{1}{x^2}f\right\|_{L^2[0,1]},
$$
then $\dom(A_{\nu_a,\infty})\supset\dom(A_{\nu,\infty})$.
\end{proof}

\begin{corollary}\label{C:Extension}

{\it{(i)}} Let $\nu\in(0,1)$. All self-adjoint extensions of the operator
${A_{\nu,\infty}}$  described by the formula
\begin{equation*}
{ A}_{{\nu,\infty}_h}={A^*_{\nu,\infty}}\upharpoonright \dom({ A}_{{\nu,\infty}_h}),~~ h
\in \mathbb{R}\cup\{\infty\};
\end{equation*}
\begin{equation}\label{E:1.356}
\dom({ A}_{{\nu,\infty}_h})=\{f\in \dom({A^*_{\nu,\infty}}):[f,x^{\frac{1}{2}-\nu}+2\nu
hx^{\frac{1}{2}+\nu}]_0=0\}.
\end{equation}

{\it{(ii)}} Let $\nu\in(0,1)$. Extension  ${ A}_{{\nu,\infty}_h}$ is  non-negative, ${
A}_{{\nu,\infty}_h}\geqslant 0$ if and only if $h\geqslant 0.$
\end{corollary}
\begin{proof}
(i) Using boundary triplet~\eqref{triple1}, we will prove the corollary, by applying
Proposition~\ref{propo}~(iii).

 (ii) By virtue of the Proposition~\ref{prkf}~(ii), ${ A}_{{\nu,\infty}_0}$ is the
 Friedrichs extension.
  From~\eqref{M0_inf_nu} it  follows that $M_{\nu,\infty}(0)=0$ and
   then, by virtue of the Proposition~\ref{prkf}~(iii), the
   extension
   ${ A}_{{\nu,\infty}_h}$ is a non-negative,
   ${ A}_{{\nu,\infty}_h}\geqslant 0$
if and only if $h\geqslant M_{\nu,\infty}(0)= 0.$
\end{proof}
%
%
\begin{theorem}\label{P:F_def}
Let $\nu\in[0,1)$ and ${ A}_{{\nu,\infty}_F}$ be the Friedrichs extension of the
operator ${ A}_{\nu,\infty}$. Also assume  $\xi\in C_0^1(\mathbb R_+)$,
$\xi(x)=\left\{\begin{array}{ll}
                 1, & x\in(0,1/2), \\
                 0, & x\geqslant 3/4.
               \end{array}
\right.$ Then:

 $(i)$ For $\nu\in(0,1)$ the quadratic form $\mathfrak{a}_{\nu,\infty}$ quadratic form
 associated with  the Friedrichs extension ${ A}_{{\nu,\infty}_F}$ takes the form
\begin{equation}\label{E:1.100e}
\mathfrak{a}_{\nu,\infty}[u]=\int
\limits_0^\infty|u'(x)|^2dx+\left(\nu^2-\frac{1}{4}\right)\int
\limits_0^\infty\frac{|u(x)|^2}{x^2}dx,
\end{equation}

\begin{equation}\label{E:1.100ae}
\dom(\mathfrak{a}_{\nu,\infty}) = H^1_0(\mathbb R_+).
\end{equation}

$(ii)$ For $\nu=0$ the quadratic form $\mathfrak{a}_{0,\infty}$ quadratic form
 associated with  the Friedrichs extension ${ A}_{{0,\infty}_F}$ takes the form

\begin{equation}\label{E:1.100ye}
\mathfrak{a}_{0,\infty}[u]=\int \limits_0^\infty\left|u'(x)-\frac{u(x)}{2x}\right|^2dx,
\end{equation}

\begin{equation}\label{E:1.100aae}
 \dom(\mathfrak{a}_{0,\infty})\supset H^1_0(\mathbb R_+)
                                     \dot{+}\mathrm{span}\left\{x^{\frac{1}{2}}\left|
\log \left(x\right)\right|^{-\alpha}\xi(x):0<\alpha\leqslant\frac{1}{2}\right\}.
\end{equation}
Wherein $\dim\left(\mathrm{dom}(\mathfrak{a}_{0,\infty})\diagup {H}_0^1(\mathbb
R_+)\right)=\infty$.

$(iii)$ For $\nu\in[0,1)$ the domain of the Friedrichs extension ${A_{\nu,\infty}}_{F}$
takes the form
\begin{equation}\label{E:F_def}
\dom({A_{\nu,\infty}}_{F})=A_{\nu,\infty}^*\upharpoonright\mathrm{dom}\dom({A_{\nu,\infty}}_{F})
,\quad
    \mathrm{dom}(A_{F}(\nu;\infty))=H_0^{2}(\mathbb R_+)\dotplus
    \mathrm{span}\{x^{\frac{1}{2}+\nu}\xi(x)\}.
\end{equation}

$(iv)$  For the quadratic form ${\mathfrak a_{\nu,\infty}}_{h}$ associated with the
operator ${A_{\nu,\infty}}_{h}$ the following decomposition is valid

\begin{equation}\label{E:116669}
\dom({\mathfrak a_{\nu,\infty}}_{h}) = H^1_0(\mathbb R_+)\dotplus
\spann\{x^{1/2-\nu}\xi(x)\},\qquad \nu\in(0,1).
\end{equation}
\end{theorem}
\begin{proof}
(i) By Hardy's inequality for $\nu\in(0,1)$, $u\in H^1_0(\mathbb R_+)$, we have
\begin{eqnarray}\label{E:Kvad_form11}
\mathfrak{a}_{\nu,\infty}[u]&=&\|u'(t)\|_{2}^2+
(\nu^2-1/4)\int\limits_0^\infty\frac{|u(t)|^2}{t^2}dt\nonumber\\
&\leqslant& \|u'(t)\|_{2}^2(1+|4\nu^2-1|), \qquad u\in H_0^1(\mathbb R_+).
\end{eqnarray}
Thus $H_0^1(\mathbb R_+)\subset\dom(\mathfrak{a}_{\nu,\infty})$.

We prove the converse inequality. Suppose firstly that $\nu\in[1/2,1)$. Then, for $u\in
H^1_0(\mathbb R_+)$,
\begin{equation}\label{E:Har_111}
\mathfrak{a}_{\nu,\infty}[u]=\|u'(t)\|_{2}^2+
(\nu^2-1/4)\int\limits_0^\infty\frac{|u(t)|^2}{t^2}dt\geqslant \|u'(t)\|_{2}^2, \quad
u\in H_0^1(\mathbb R_+).
\end{equation}

If $\nu\in(0,1/2)$, then for $u\in H_0^1(\mathbb R_+)$ applying the Hardy's inequality
we obtain
\begin{eqnarray}\label{E:Kvad_form22}
\mathfrak{a}_{\nu,\infty}[u]&=&\|u'(t)\|_{2}^2-
(1/4-\nu^2)\int\limits_0^\infty\frac{|u(t)|^2}{t^2}dt\nonumber\\
&\geqslant& \|u'(t)\|_{2}^2+(4\nu^2-1)\|u'(t)\|_{2}^2=4\nu^2\|u'(t)\|_{2}^2.
\end{eqnarray}

So, the energy norm of ${A_{\nu,\infty}}$ on $H_0^1(\mathbb R_+)$  is equivalent to the
norm of the space $H_0^1(\mathbb R_+)$. Since $H_0^2(\mathbb
R_+)=\dom({A_{\nu,\infty}})$ is dense in the energy space of the operator
${A_{\nu,\infty}}$, then $\dom(\mathfrak{a}_{\nu,\infty})$ and $H_0^1(\mathbb R_+)$
coincide algebraically and topologically.

(ii) Let $u_\alpha(x)=x^{\frac{1}{2}}\left| \log (x)\right|^{-\alpha}\xi(x)$, then
$$
\mathfrak{a}_{0,\infty}[u_\alpha]=\int
\limits_0^{1/2}\left|u_\alpha'(x)-\frac{u_\alpha(x)}{2x}\right|^2dx=-\frac{\alpha^2
2^{2\alpha+1}}{2\alpha+1}.
$$
So $\{x^{\frac{1}{2}}\left|
\log(x)\right|^{-\alpha}\xi(x)\}\subset\dom(\mathfrak{a}_{0,\infty})$.

Let functions $x^{\frac{1}{2}}\left| \log (x)\right|^{-\alpha}\xi(x)$ are linearly
independent.

Conversely, $\sum\limits_{j=1}^n C_j x^{\frac{1}{2}}\left| \log
(x)\right|^{-\alpha_j}\xi(x)=0$, for $\alpha_j\in(0,\frac{1}{2}]$, $x\in(0,1)$. We order
degrees: \linebreak$\alpha_1<\alpha_2<\ldots<\alpha_n$. Then multiplying by the term
with the smallest degree, we obtain
$$
C_1+\sum\limits_{j=2}^n \left| \log (x)\right|^{-\alpha_j+\alpha_1}=0.
$$
Thus $C_1=0$.

Similarly, we obtain that $C_j=0$. This is a contradiction.

 (iii) We note that $H_0^2(\mathbb{R}_+)\subset H_0^1(\mathbb{R}_+)$. If $u(x)=x^{1/2+\nu}\xi(x)$
then $u'(\cdot)\in L^2(\mathbb  R_+)$, but
$u(\cdot)\not\in\dom({A_{\nu,\infty}})=H_0^{2}(\mathbb R_+)$. By the construction of the
Friedrichs extension and the equalities~\eqref{dom*}, we obtain
        \begin{equation*}
\begin{gathered}
\dom({A_{\nu,\infty}}_{F})=\dom({A^*_{\nu,\infty}})\cap
\dom(\mathfrak{a}_{\nu,\infty}[u])=\dom({A^*_{\nu,\infty}})
\cap H^1_0(\mathbb R_+)=\\
=H_0^{2}(\mathbb R_+)\dotplus \spann\{x^{{1/2+\nu}}\xi(x)\}.
\end{gathered}
\end{equation*}

(iv) The proof follows from~\cite[Theorem 1]{Malamud} and from the fact that
$x^{1/2+\nu}\xi(x)\in H_0^1(\mathbb R_+)$.

\end{proof}
%
%
%
%

\begin{remark}
Theorem~\ref{P:F_def} strengthens and complements the results of the works~\cite{Kalf}
and~\cite{BDG}. For example, for $\nu\in(0,1)$ in~\cite{BDG} it is only shown that
$\mathrm{dom}({A_{\nu,\infty}}_{F})$ is dense in $H_0^1(\mathbb R_+).$
\end{remark}

\begin{remark}\label{P:F_F}
Note that the domains of the Friedrichs extensions in~\eqref{Fri_fi} and~\eqref{E:F_def}
coincide.
\end{remark}

Indeed, since $[f, x^{\frac{1}{2}+\nu}]_0=0$, then for $f=x^{\frac{1}{2}}\xi(x)$, we
obtain
$$
[f,
x^{\frac{1}{2}+\nu}]_0=\lim\limits_{x\to0}\left(\left(\frac{1}{2}+\nu\right)x^{2\nu}\xi(x)-
\left(\frac{1}{2}+\nu\right)x^{2\nu}\xi(x)-x^{1+2\nu}\xi'(x)\right)=0,
$$
where $\xi\in C_0^1(\mathbb R_+)$, $\xi(x)=1$ for $x\in[0,1]$.

\section{Connection of the Weyl functions of the operators ${A_{\nu,b}}$ and ${A_{\nu,\infty}}$}
\begin{proposition}\label{P:A_b}
Let ${A_{\nu,b}}$ and ${A_{\nu,\infty}}$ be the operators with domains~\eqref{Dirihle}
and~\eqref{Dom}, respectively. Assume that $\Pi_{\nu,b}$ and $\Pi_{\nu,\infty}$ be the
boundary triplets of the form~\eqref{triple} and~\eqref{triple1},  $M_{\nu,b}(z)$ and
$M_{\nu;\infty}(z)$ be the Weyl functions of the form~\eqref{Weyl_fun}
and~\eqref{Weyl_fun1}. Then the relation
$$
\lim\limits_{b\rightarrow+\infty}M_{\nu,b}(z)=M_{\nu;\infty}(z)
$$
holds uniformly on compact subsets of $\mathbb C_+$.
\end{proposition}
\begin{proof}
First, we consider the case $\nu\in(0,1)$.  Since the Bessel functions $J_{\nu}(t)$ and
$J_{-\nu}(t)$ for $t\rightarrow\infty$ have the asymptotic behavior \eqref{asym_infty},
we have
 \[\lim\limits_{b\rightarrow+\infty}M_{\nu,b}(z)=-\lim\limits_{b\rightarrow+\infty}
 \frac{\Gamma(1-\nu)}{2\nu4^{\nu}\Gamma(1+\nu)}\cdot\frac{J_{-\nu}(b\sqrt{z})}{J_{\nu}(b\sqrt{z})}\cdot
z^{\nu}=
\]
\[=-\lim\limits_{b\rightarrow+\infty}
\left[\frac{\Gamma(1-\nu)}{2\nu4^{\nu}\Gamma(1+\nu)}\cdot\frac{\cos\left(b\sqrt{z}+\frac{\nu\pi}{2}-\frac{\pi}{4}\right)}
{\cos\left(b\sqrt{z}-\frac{\nu\pi}{2}-\frac{\pi}{4}\right)}\cdot{z^{\nu}}\right]=
\]
\[=-\frac{\Gamma(1-\nu)}{2\nu4^{\nu}\Gamma(1+\nu)}\lim\limits_{b\rightarrow+\infty}
\frac{e^{-i(b\sqrt{z}+\frac{\nu\pi}{2}-\frac{\pi}{4})}}
{e^{-i(b\sqrt{z}-\frac{\nu\pi}{2}-\frac{\pi}{4})}}\cdot
z^{\nu}=e^{i(1-\nu)\pi}\frac{\Gamma(1-\nu)}{2\nu4^{\nu}\Gamma(1+\nu)}\cdot{z^{\nu}}=M_{\nu;\infty}(z).
\]

The case $\nu=0$ is treated similarly. Namely,
\begin{eqnarray*}
&{}&\lim\limits_{b\rightarrow+\infty}M_{0,b}(z)=\lim\limits_{b\rightarrow+\infty}\left[
-\log\left(\frac{\sqrt{z}}{2}\right)+\frac{\pi}{2}\frac{Y_0(b\sqrt{z})}{J_0(b\sqrt{z})}-\gamma\right]\\
=\lim\limits_{b\rightarrow+\infty}&{}&\left[-\log\left(\frac{\sqrt{z}}{2}\right)+\frac{\pi}{2}\cdot
\frac{\sin(b\sqrt{z}-\frac{\pi}{4})}{\cos(b\sqrt{z}-\frac{\pi}{4})}-\gamma\right]=-\log\left(\frac{\sqrt{z}}{2}\right)
+\frac{\pi i}{2}-\gamma=M_{\nu;\infty}(z).
\end{eqnarray*}
It is easily seen that the convergence in both relations is uniform on compact subsets.
\end{proof}

\section{Singular Sturm-Liouville operators of the Bessel type}
Here, we consider the Sturm-Liouville differential expression
\begin{equation}\label{expres}
\tau u:=-u''+qu
\end{equation}
in $L^2(\mathbb{R}_+)$ with certain potentials $q$.

The minimal operator $T_{\min}=T$ associated with \eqref{expres} is the closure of the
operator $T'$ of the form
\begin{equation}\label{min}
T'u:=\tau u, \quad \dom(T')=\{u:u\in\mathfrak{D},
u\ \text{has the compact support in}\ (0,\infty)\},
\end{equation}
where
\begin{equation}\label{D}
\mathfrak{D}:=\{u: u\in AC_{\mathrm{loc}}(\mathbb{R}_+)\cap L^2(\mathbb{R}_+), u'\in
AC_{\mathrm{loc}}(\mathbb{R}_+), \tau u\in L^2(\mathbb{R}_+)\},
\end{equation}
and $T$ is a densely defined symmetric operator.

The maximal operator associated with \eqref{expres} is
\begin{equation}\label{max}
T_{\max}=T^*=\tau\upharpoonright\mathfrak{D}.
\end{equation}

The following relations hold:
$$
T_{\min}=T=\overline{T'}=T^{**}=T_{\max}^*.
$$

\begin{corollary}\label{1}
Let $q\in L^1_{\mathrm{loc}}(\mathbb{R}_+)$ and
\begin{equation}\label{q1}
q(x)\geq\frac{\beta}{x^2}-\mu, \quad (x\in\mathbb{R}_+)
\end{equation}
for some $\beta>-\frac{1}{4}$ and $\mu\geq0$. Then
\begin{itemize}
\item[\textit{(i)}] The closure $\mathfrak{t}_q$ of the quadratic form
  $\mathfrak{t}'_q$  associated with the operator $T$ is
\begin{equation}\label{form}
\begin{gathered}
\mathfrak{t}_q[u]=\int\limits_0^{\infty}|u'(x)|^2dx+\int\limits_0^{\infty}q(x)\cdot|u(x)|^2dx,\\
\qquad \dom(\mathfrak{t}_q)=\{u\in
H^1_0(\mathbb{R}_+):\int\limits_0^{\infty}q(x)\cdot|u(x)|^2dx<\infty\}=:H^1_0(\mathbb{R}_+;q).
\end{gathered}
\end{equation}
\item[\textit{(ii)}]\cite{Kalf} The domain of the Friedrichs extension $T_F$ of $T$ is
\begin{equation}\label{fried}
\dom(T_F)=\mathfrak{D}\cap H^1_0(\mathbb{R}_+;q),
\end{equation}
where $\mathfrak{D}$ is given by \eqref{D}.
\end{itemize}
\end{corollary}
\begin{proof}
Without loss of generality, we can assume that $\mu=0$. Let
$\beta=\nu^2-\frac{1}{4}>-\frac{1}{4}.$
 Consider the quadratic form $\mathfrak{t}_q$ associated with the operator
$T_F$. Since $q(x)>\frac{\nu^2-\frac{1}{4}}{x^2}$, we have
\begin{equation}\label{*}
\dom(\mathfrak{t}'_q)\subset\dom(\frak a_{\nu,\infty})=H^1_0(\mathbb{R}_+),
\end{equation}
where $a_{\nu,\infty}$ is given by \eqref{E:1.100}.


Further, let $u(\cdot)\in C_0^\infty(\mathbb{R}_+)\subset\dom(T')$. Integrating by
parts, we obviously have
\begin{equation}\label{form_proof}
\begin{gathered}
\mathfrak{t}'_q[u]=(Tu,u)=\lim\limits_{x\rightarrow\infty}\left[u'(t)u(t)|_0^{x}+\int\limits_0^{x}|u'(t)|^2dt+
\int\limits_0^{x}q(t)\cdot|u(t)|^2dt\right]\\=\int\limits_0^{\infty}|u'(x)|^2dx+\int\limits_0^{\infty}q(x)\cdot|u(x)|^2dx.
\end{gathered}
\end{equation}
Taking the closure of these forms and~\eqref{*} into account, we arrive at~\eqref{form}.

According to the construction of the Friedrichs extension and~\eqref{D}, we get
\begin{equation*}
\dom(T_F)=\dom(T^*)\cap \dom(\mathfrak{t}_q)=\mathfrak{D}\cap H^1_0(\mathbb{R}_+;q).
\end{equation*}
The Corollary is proved.
\end{proof}

\begin{corollary}
Let $q\in L^1_{\mathrm{loc}}(\mathbb{R}_+)$ and
\begin{equation}\label{q2}
q(x)\geq-\frac{1}{4x^2}-\mu, \quad (x\in\mathbb{R}_+)
\end{equation}
for some $\mu\geq0$.
 Then
\begin{itemize}
\item[\textit{(i)}] The closure $\mathfrak{t}_q$ of the quadratic form
  $\mathfrak{t}'_q$  associated with the operator $T$ takes the form
\begin{gather}\label{form'}
\mathfrak{t}_q[u]=\int\limits_0^{\infty}|u'(x)|^2dx+\int\limits_0^{\infty}q(x)\cdot|u(x)|^2dx,\\
\quad \dom(\mathfrak{t}_q)=
{H}^1_0(\mathbb{R}_+;q)\dot{+}\mathrm{span}\{x^{1/2}\xi(x)\},
\end{gather}
where $\xi\in C_0^1(\mathbb R_+)$, $\xi(x)=1$, for $x\in[0,1]$.
\item[\textit{(ii)}]\cite{Kalf} The domain of the Friedrichs extension of $T$ is
\begin{equation}\label{fried'}
\dom(T_F)=\mathfrak{D}\cap
({H}^1_0(\mathbb{R}_+;q)\dot{+}\mathrm{span}\{x^{1/2}\xi(x)\}),
\end{equation}
where $\mathfrak{D}$ is given by \eqref{D}.
\end{itemize}
\end{corollary}

The proof is similarly to Corollary~\ref{1}.

\begin{corollary}
Another description of the Friedrichs extension
was obtained by H.~Kalf in~\cite{Kalf}
\begin{equation*}
\dom(T_F)=\left\{u:u\in\mathfrak D,\
\int\limits_0^\infty\left|u'-\frac{u}{2x}\right|<\infty\right\}.
\end{equation*}
\end{corollary}

\textbf{Acknowledgments.} The authors express their gratitude to M.M.~Malamud for posing
the problem and the permanent attention to the work, and to A.S.~Kostenko for useful
discussions and remarks.



\begin{thebibliography}{99}



\bibitem {A&S}
        \emph{Handbook of Mathematical Functions With Formulas, Graphs, and Mathematical
        Tables}
        / ed. M. Abramowitz, I.A. Stegun, United States Department
        of Commerce, Washington D.C., 1972.

\bibitem{AG}
       N.I. Akhiezer and I.M. Glazman, \emph{Theory of Linear Operators in Hilbert Space} –
Volume II, Pitman, London, 1981; reprinted: Dover, New York, 1993.

\bibitem{AA}
       V.S.  Alekseeva, A.Yu. Ananieva,  \emph{On extensions of the Bessel operator on a finite interval
       and a half-line}, J. of Math. Scien. \textbf{187} (2012) 1, 1--8.

\bibitem{BDG}
       L. Bruneau, J. Derezi\'{n}ski, V. Georgescu,  \emph{Homogeneous Schr\"{o}dinger
        Operators on Half-Line}, Ann. Henri Poincar\'{e} \textbf{12} (2011),
547--590.
\bibitem{CG}
      S. Clark, F. Gesztesy, R. Nichols,  \emph{Principal Solutions Revisited},
       arXiv.org:1401.1285.
\bibitem{DM_91}
       V.A.  Derkach, M.M. Malamud,  \emph{Generalized rezolvent and the boundary value
problems for Hermitian operators with gaps}, J. Funct. Anal. \textbf{95} (1991) 1,
1--95.
\bibitem{DerMal95}
V.A. Derkach, M.M. Malamud,  \emph{The extension theory of Hermitian operators and the
moment problem}, J. Math. Sci. (New York) \textbf{73} (1995), 141--242.
\bibitem{EGNT}
        J. Eckhardt, F. Gesztesy, R. Nichols, G. Teschl,  \emph{Weyl--Titchmarsh Theory
for Sturm--Liouville Operators With Distributional Potentials}, Opuscula Math.
\textbf{33} (2013) 3, 467--563.

\bibitem{EvKalf}
         W.N. Everitt, H. Kalf,  \emph{The Bessel differential equation and the
         Hankel transform}, Jour. of Comput. and App. Math. \textbf{208} (2007), 3--19.

 \bibitem{Ful} C. Fulton, \emph{Titchmarsh--Weyl $m$--functions for second--order Sturm--Liouville
  problems with two singular endpoints}, Math. Nachr. \textbf{281} (2008), No. 10, 1418--1475.

\bibitem{FulLag} C. Fulton, H. Langer, \emph{Sturm--Liouville operators with singularities and
generalized Nevanlinna functions}, Complex Analysis and Operator Theory. \textbf{4}
(2010), No. 2, 179--243.

\bibitem{GG}
   V.I. Gorbachuk,  M.L. Gorbachuk, \textit{Boundary Value Problems for
Operator Differential Equations}, Mathematics and its Applications (Soviet Series) 48,
Kluwer Academic Publishers Group, Dordrecht, 1991.
\bibitem {Oliveira} Cesar R. de Oliveira, \emph{Intermediate
        Spectral Theory and Quantum Dynamics,} Birkhauser, Berlin, 2000.
\bibitem{Kalf}
         H. Kalf,  \emph{A Characterization of the Friedrichs Extension of
          Sturm-Liouville Operators}, J. London Math. Soc. \textbf{17} (1978), No. 2, 511--521.
\bibitem{Ka}
        T. Kato, \emph{Perturbation Theory for Linear Operators.} Springer, Berlin, 1995.

\bibitem {Kochybei}
        A. N. Kochybei, \emph{Self-adjoint extensions of the SchrЁodinger operator with
        singular potential}, Sib. Mat. Zh. \textbf{32} (1991), No. 3, 60--69.

\bibitem {KST} {A.~Kostenko, G.~Teschl} \emph{On the singular Weyl--Titchmarsh function
of perturbed spherical
 Schr\"odinger operators}, J. Diff. Eqs. \textbf{250} (2011), 3701--3739.

\bibitem{Malamud}
M.M. Malamud,  \emph{On some classes extensions of the Hermitian operators with gaps},
UMJ \textbf{44} (1992), No. 2, 215--233.

\bibitem {Naj69}
        M. A. Naimark, \emph{Linear Differential Operators}, Nauka, Moscow, 1969.

\bibitem{Rid}
        M. Reed and B. Simon, \emph{Methods of Modern Mathematical Physics,} Academic Press,
        Inc., New York, 1972.
\bibitem {Stein}
        E. M. Stein, \emph{Singular Integrals and Differentiability Properties of
        Functions,} Mir, Moscow, 1973.

\end{thebibliography}
\end{document}